\documentclass[12pt]{amsart} 

\usepackage{graphicx,color}
\usepackage{amsmath,amssymb,amsthm,
amscd,amsopn,amsfonts,amsxtra,hyperref}
\usepackage{latexsym,array,verbatim,centernot,mathtools,cancel}
\usepackage[mathscr]{eucal}
\usepackage{eucal}
\usepackage{eufrak} 

 \usepackage{mathrsfs}

\newtheorem{theorem}[subsection]{Theorem}
\newtheorem{lemma}[subsection]{Lemma}
\newtheorem{proposition}[subsection]{Proposition}
\newtheorem{corollary}[subsection]{Corollary}
\newtheorem{hypothesis}
{Hypothesis}

\theoremstyle{definition}
\newtheorem{definition}[subsection]{Definition}

\theoremstyle{remark}
\newtheorem{remark}[subsection]{Remark}

\newcommand{\norm}[1]{\left\| #1\right\| }

\def\al{\alpha}
\def\de{\delta}

\def\ga{\gamma}

\def\veps{\varepsilon}
\def\vphi{\varphi}
\def\lam{\lambda}
\def\om{\omega}
\def\De{\Delta}

\def\Om{\Omega}

\def\sign{\mathrm{sign}}

\def\iy{\infty}

\def\pa{\partial}

\def\sH{\mathscr{H}}
\def\To{\Rightarrow}

\newcommand{\crr}{{\color{red}$\dagger$}}

\newcommand{\dive}{\mathrm{div}\,}
\newcommand{\la}{\langle}
\newcommand{\ra}{\rangle}

\newcommand{\hr}{\hookrightarrow}

\newcommand{\inv}{^{-1}}
\newcommand{\td}{\tilde}
\newcommand{\wtd}{\widetilde}

\newcommand{\nd}{\noindent}
\newcommand{\n}{\newline}

\newcommand{\vs}{\vspace}

\newcommand{\Z}{\mathbb{Z}}
\newcommand{\R}{\mathbb{R}}

\begin{document}

\title[Threshold for NLS with Rotation]{Threshold for Blowup and Stability for Nonlinear Schr\"odinger Equation with Rotation}
\author{Nyla Basharat, Hichem Hajaiej, Yi Hu and Shijun Zheng}
\address[Nyla Basharat]{Department of Mathematics and Statistics\\
University of Saskatchewan, Saskatoon, SK  S7N 5E6, Canada}
\email{nylabasharat12@gmail.com}
\address[Hichem Hajaiej]{Department of Mathematics, California State University, Los Angeles, CA 90032}
\email{hhajaie@calstatela.edu} 
\address[Yi Hu]{Department of Mathematical Sciences\\
        Georgia Southern University, Statesboro, GA 30460}
\email{yihu@GeorgiaSouthern.edu}
\address[Shijun Zheng]{Department of Mathematical Sciences\\
        Georgia Southern University, Statesboro, GA 30460}
\email{szheng@GeorgiaSouthern.edu}
\date{}


\vspace{.1260in}
\setcounter{equation}{0}

\maketitle 



\maketitle

\begin{abstract} 
We consider the focusing NLS with an angular momentum  and a harmonic potential, 
which models  Bose-Einstein condensate under a rotating magnetic trap. 
We give a sharp condition on the global existence and blowup in  the mass-critical case. 
We further consider the stability of such systems  via variational method. 
We  determine that at the critical exponent $p=1+4/n$,
the mass of $Q$, the ground state for the  NLS with zero potential, is the threshold for both finite time blowup and orbital instability. Moreover, we prove 
 a sharp threshold theorem  for the  rotational NLS with an inhomogeneous nonlinearity.  
 The analysis relies on the existence of ground state as well as a virial identity for  
the associated kinetic-magnetic operator. 
\end{abstract}


\vs{.20in} 

\tableofcontents

\section{Introduction}\label{s:intro} 
The Cauchy problem for the nonlinear Schr\"odinger equation (NLS), or Gross-Pitaevskii equation with rotation in $(t, x) \in \mathbb{R}^{1+n}$, $n\ge 2$ reads 
	\begin{align}\label{e:nls-H-rot}
	&i u_t = -\frac12\Delta u + V(x) u -\lam |u|^{p-1} u +L_\Om u \\ 
	 &   u(0) = u_0 \in \mathscr{H}^1,\notag
	\end{align}
where $p\in [1,1+4/(n-2))$, 
$V(x) := \frac12\gamma^2 |x|^2$ is the harmonic potential with frequency $\gamma>0$ that models the magnetic trap and
$\lam $ is a real constant. 
Adding a spin symmetry, the angular momentum operator is denoted by $L_\Omega:=-\Om L_z=iA\cdot \nabla$, 
where $\Om\in\R$ is the 
 rotation frequency, $L_z = i(x_2 \partial_{x_1} - x_1 \partial_{x_2})$ and $A=\Om\la -x_2,x_1,0,\dots,0\ra$.  
The weighted Sobolev space is given by $\sH^{s}=\sH^{s,2}$,  
where 
	\begin{align*}
	\mathscr{H}^{s,r} := \{ f \in L^r(\mathbb{R}^n): (-\De)^{s/2} f \in L^r, \la x\ra^s f \in L^r \},
	\end{align*} 
with $\la x\ra=(1+|x|^2)^{1/2}$. 
Then $\sH^1=\Sigma:=H^1\cap L^2(|x|^2dx)$
arises naturally as a suitable energy space in this setting, where $H^1$ is the usual Sobolev space.  

When $n=2,3 $, the NLS above models Bose-Einstein condensation with rotation \cite{Af06,
BaoCai13,BaoWaMar05,BuRo99,MAHHWC99,Sei02}, 
 which is a remarkable system arising in 
optics, 
{plasma}, superfluids, spinor particles, quantized vortices and surface waves. 
Extensions to higher dimensions can be found in \cite{BHZ19a,Gar12}. 
Mathematically, it can be viewed as the mean field limit   of rotating many-body bosons in a confining trap
\cite{LeNR14,LieS06}. 
 Equation \eqref{e:nls-H-rot} can be formally derived from
	\begin{align*}
  i \frac{\partial u}{\partial t} = \frac{\delta H}{\delta \bar{u}},
	\end{align*}
where $H$ is the associated Hamiltonian
	\begin{align*}
	H[u] = \int \left( \frac12|\nabla u|^2 + V |u|^2 -\frac{2 \lam}{p+1} |u|^{p+1} -\Om \, \overline{u} L_z u \right).
	\end{align*}
It is desirable to provide  rigorous mathematical descriptions for the rotating BEC model. 
As is known, the well-posedness and blowup for the  standard NLS 
have been extensively studied for a few decades  in the energy sub-critical regime $p<1+\frac{4}{n-2}$
and  the  energy critical regime $p=1+\frac{4}{n-2}$, see e.g. the expository in \cite{Tao2006}. 
For NLS with a harmonic potential, the analogue were considered in e.g.  \cite{Car02a,KVZ2009,Zh00}.
For the rotational NLS (RNLS) in (\ref{e:nls-H-rot}), 
the local wellposedness (l.w.p.)  
has been considered in  e.g. \cite{AMS12,CazE88,HHsiaoL07b} when $n=2, 3$, and  \cite{deB91,GaZ13a,Mi08,Z12a}  for general magnetic NLS (mNLS), to list  a few. 
RNLS (\ref{e:nls-H-rot}) can be written in the magnetic form (\ref{e:ut-L-p_AV}), 
thus the wellposedness follows from that of  mNLS if $p\in (1, 1+4/(n-2))$ 
in view of Proposition  \ref{pr:gwp-lwpUav}. 

In the focusing case, i.e., $\lam>0$, if  $p<1+{4}/{n}$, it is known that all $\sH^1$-solutions of equation (\ref{e:nls-H-rot}) exist globally in time. However, if $\lam>0$ and 
$p\in [1+\frac{4}{n},1+\frac{4}{n-2})$, there exist finite time blowup solutions for \eqref{e:nls-H-rot}, see  e.g. \cite{AMS12,BHIM,Car02a,Gon91b}.  
So, notably $p=1+4/n$ is  the mass-critical case, where 
the situation becomes more subtle and complex, and the occurrence of global  existence or wave collapse  
 depends on both the size and 
  the profile of the initial data. 
  Thus, it is an open question to find a threshold condition that distinguishes the g.w.p. and blowup for the focusing RNLS. 

 
 In this paper, we are mainly concerned with the focusing mass-critical case $p=1+4/n$.  
Inspired by the work in \cite{Wein83} and \cite{Zh00},   
we address this  so-called ``minimal mass blowup''  problem for \eqref{e:nls-H-rot}.  
To our knowledge, previous techniques and ideas  do not directly apply to solving this threshold problem.  
One delicate issue is that the threshold profile for \eqref{e:nls-H-rot} turns out {\em not} to be the minimizer for the associated energy!  
  This is one essential reason that makes the problem more challenging.   
  Technically, in the presence of $V$ and $L_\Om$,
some of the key symmetries for the standard NLS (including scaling  and translation invariance) 
are broken, whose geometry of the trajectories of motion of particles has  the effect on relevant physical phenomenon of wave collapse as well as stability of solitons.  
 We wish to note that all the results in this paper have  extensions to the case where $\lam$ is a variable in $x$ 
and $L_\Omega$ is substituted with $L_A:=iA\cdot \nabla$ where $A=Mx$, 
$M$ being a skew-symmetric matrix, see \cite{BHZ19a}.  
In physics setting, such $A$  satisfies the Poincar\'{e} gauge  ($x\cdot A=0$) 
and the Coulomb gauge  ($\dive(A)=0$) conditions. 
 For expository reason, we  first present our  theorems (Theorems \ref{t:gwp-blowup-Om} and \ref{t:stabi-masscr3}) concerning the threshold  for blowup and stability 
 in the simpler case where $\lam>0$ is a  constant and $A$ is given as in (\ref{A:Om-x}).  
 Then it might elucidate the treatment along with the proofs of the analogous results in Theorems \ref{sharpQ:inhom} and \ref{t:orb-stab-Zc}
 for the inhomogeneous case. 
    
In our first main theorem, we give a sharp threshold in terms of the unique positive radial ground state $Q=Q_\lam=Q_{\lam,1}$  with $\lam>0$
in  $H^1(\R^n)$:  
\begin{align}
& -\frac12\De Q-\lam |Q|^{p-1}Q =-Q.\label{e:grstQ-masscri}
\end{align} 
The existence of $Q$ is well-known \cite{Kw89,Me96,Wein83}.  

\begin{theorem}[threshold for g.w.p. and blowup]\label{t:gwp-blowup-Om} Let $p=1+\frac{4}{n}$ and $\lam>0$. 
Let $(\Om, \ga)$ be any given  pair in $\R\times\R_+$.  
Suppose $u_0\in \sH^1$. 
\begin{enumerate} 
\item[(a)] If $\Vert u_0\Vert_2< \Vert Q\Vert_2$, then there exists a unique global in time solution $u$ of \eqref{e:nls-H-rot} in $C(\R,\sH^1)\cap L^{2+\frac{4}{n}}_{loc}(\R, L^{2+\frac{4}{n}})$. 
\item[(b)] The condition in (a) is sharp in the sense that for all $c\ge \Vert Q\Vert_2$, 
there exists $u_0$ in $\sH^1$ satisfying 
 $\Vert u_0\Vert_2= c$ such that \eqref{e:nls-H-rot} has a finite time blowup solution.  
\end{enumerate}
\end{theorem}

 When $n=2$, the threshold value in the theorem can be evaluated at $\lam^{-\frac{1}{2}}\Vert Q_{1,1}\Vert_2=\lam^{-\frac{1}{2}} (\pi \cdot 1.86225\cdots)^{\frac{1}{2}}$, 
 see \cite{GuoSeir14,Wein83}. 
For blowup solutions, in general, the wave collapse  depends on the delicate balance 
between kinetic and potential energies (linear vs. nonlinear) as well as angular momentum for the profile of the solution, under the constriction of mass and energy conservation. 

Theorem \ref{t:gwp-blowup-Om}  shows that \eqref{e:nls-H-rot} has the same minimal mass $\norm{Q}_2^2$ for blowup  
 as in the free case $\Om=V=0$.  
The heuristic reason is  that in 
\eqref{e:nls-H-rot}  neither $V$ nor $-\Om L_z$  sees the scaling.  Note that, however, at the threshold, $u_0=Q$
leads to a soliton solution in the absence of potential while the same initial data leads to blowup solution for \eqref{e:nls-H-rot}. 

The proof of Theorem \ref{t:gwp-blowup-Om} is given in Section 5, which mainly
relies on a sharp criterion in part (a) of Lemma  \ref{l:J-J'-blup}, 
 where  we prove the blowup for \eqref{e:nls-H-rot}  for all $\Om$, $\ga$  if $p=1+{4}/{n}$, 
namely,  $u$ blows up in finite time provided $E_{0,0}(u_0)=\frac12\int |\nabla u_0|^2 -\frac{\lam n}{n+2} \int |u_0|^{2+\frac{4}{n}}\le 0$. 
This lemma is sharp in the sense that it dictates the blowup for initial data $u_0=Q$ with $E_{0,0}(Q)=0$ 
such that the $Q$-blowup profile is attainable,  
which is not covered by 
existing results in the literature as far as the authors know. 
 In Lemma \ref{l:J-J'-blup} (c)-(d) we also obtain some blowup conditions if $p>1+{4}/{n}$, 
 which allows us to show in Proposition \ref{p:blowup-masssup} a blowup result above the ground sate level. 
 Our  approach is mainly motivated by the treatment in \cite{Wein83} and \cite{Zh00}, 
where blowup results were proven with the same minimal mass in the cases  $\Om=V=0$ and $\Om=0$, $V=|x|^2$ respectively. See also related discussions 
in \cite{Car02a,Car02c}  and \cite{deB91}. 

Observe that the threshold in Theorem \ref{t:gwp-blowup-Om} 
  is valid 
  if $V$ is an isotropic harmonic potential. 
When $V$ is anisotropic, it remains an open question 
in the ``fast rotating''  regime, where $|\Om|>\underline{\ga}:=\min_{1\le j\le n}{(\ga_j)}$, 
concerning  the threshold on finite time blowup for either $p=1+4/n$ or $p>1+4/n$,
 see some numerical results in \cite{BaoCai13,BaoWaMar05}. 
In  physics, this question concerns  the scenario where the rotation frequency is stronger than the trapping frequecy,
in which case  it would be worthwhile to study the behavior of minimal mass wave collapse quantitatively. 
 
Here we would like to briefly review the pervious work on blowup for RNLS. 
The initial result  on the blowup for (\ref{e:nls-H-rot}) was obtained in \cite{Gon91b} in the case  
$ | \Omega|=\ga$,  $n=3$ if $p> 7/3$.  Recent blowup results were proven in \cite{AMS12}
in the case for all $\Om, \ga$ if $p\ge 1+4/n$, $n=2,3 $ and in \cite{BHIM,Gar12} 
for general electromagnetic potentials.  However, these results only give a  sufficient condition, 
not a {\em sharp criterion}  to address the threshold, or minimal mass blowup problem. The pre-existing results mainly assume either $E_{\Om,V}(u_0)$ or
$E_{0,V}(u_0)$ is negative,  
from which follows the blowup of the solution based on a virial identity, 
a convexity argument for the variance $J(t)=\int |x|^2 |u|^2$.  
Such assumption seems too strong in the mass-critical case. 
Notice that some generic data does not satisfy this condition: e.g. $E_{\Om,V}(Q)=E_{0,V}(Q)>0$.  
This example indicates that we need a sharper blowup condition. 
  Lemma \ref{l:J-J'-blup} thus provides such criteria with weaker conditions. 
The proof of the lemma is based on the magnetic virial identity (Lemma \ref{l:virial-J(t)-Vom}) 
 and an explicit solution of the o.d.e. (\ref{e:J''+J-EL0p}) for $J(t)$. 
Then, the blowup time can be located by examining the first zero of $J(t)$ in a more accurate way. 
Consequently, the lemma enable us to apply the sharp Gagliardo-Nirenberg inequality involving the ground 
 state given by (\ref{e:grstQ-masscri}) to prove  Theorem \ref{t:gwp-blowup-Om}
  for  all $(\Omega, \gamma)$  in all dimensions.  
  Note that if $p=1+4/n$, by evaluating the derivative of $J(t)$ at its zero $t=T_*$, we show in passing  
  a lower bound of the blowup rate $\Vert \nabla u\Vert_2\gtrsim  \Vert u_0\Vert_2(T_*-t)^{-1/2}$   in (\ref{deU-T-t}).  

Concerning the blowup rate (\ref{deU-T-t}) shown in the proof for RNLS (\ref{e:nls-H-rot}),  we add that 
when $p=1 +{4}/{n}$  and $\norm{u_0}_2=\norm{Q}_2$, 
like in the  cases  where $\Om=V=0$ and $\Om=0$, $V=|x|^2$ in \cite{Me93} and \cite{Car02c} respectively, via the 
$\mathcal{R}$-transform, a pseudo-conformal type transform,  we  
were able to determine the profile for all blowup solutions with minimal mass at the ground state level.
Hence all such blowup solutions in $\sH^1$ have blowup rate $(T_*-t)\inv$, which is however unstable, see \cite[Proposition 4.5]{BHZ19a}. 
When $p=1+{4}/{n}$ and $\norm{u_0}_2$ is slightly greater than $\norm{Q}_2$, 
N.B., Y.H. and S.Z.  proved  in \cite{BHZ19a} the $\log$-$\log$ law for RNLS (\ref{e:nls-H-rot}), 
i.e., there exists a universal constant $\al^*:=\al_n^*>0$ such that 
if $ \int |Q|^2 < \int |u_0|^2  < \int |Q|^2 + \alpha^*$ 
with negative energy  $E_{0,0}(u_0)<0$, 
then  $u\in C([0, T_*); \mathscr{H}^1)$  is a blowup solution to \eqref{e:nls-H-rot} on its lifespan $[0, T_*)$ satisfying
 \begin{align}\label{log-log_RNLS}
	\Vert \nabla u(t,\cdot) \Vert_2
	= \frac{\Vert\nabla Q\Vert_2 }{\sqrt{2 \pi}}\sqrt{ \frac{\log \left| \log (T_*- t)  \right|}{T_*-t} }+o(1) ,\quad as\; t\to T_*
	\end{align}
 where $Q$ is the unique positive radial solution of \eqref{e:grstQ-masscri}. The analogue for the standard NLS
was initially proven in Merle and Rapha\"el \cite{MR05} under the hypothesis of certain spectral property conjecture.  
We can now state the $\log$-$\log$ law  (\ref{log-log_RNLS}) for RNLS \eqref{e:nls-H-rot} 
when $n\le 12$, in light of the recent numerical verification of the spectral property conjecture in \cite{YangRouZ18}.  
Such a blowup rate is also known to be stable in $\sH^1$. 
An example of  the initial data is given in the form $u_0 =\eta Q$ with $1 < \eta<\sqrt{1+ \frac{\alpha^*}{\| Q \|_2^2}}\,$, which can be easily  verified to satisfy  the condition above by the Pohozaev identity \eqref{e:pohoz-Qlam}.

 In  the second component of this paper, in Section \ref{s:Virial-inhom}  we further consider the threshold problem for the inhomogeneous RNLS (\ref{e:psi-VomLz}) in $\R^{1+n}$,
\begin{equation}\label{e:psi-VomLz}
	iu_t = -\frac12\Delta u +\frac{\gamma^2}{2} |x|^2 u - \lam(x) |u|^{p-1} u -\Om L_z u\,, \quad u_0\in \Sigma=\sH^1
	\end{equation}
where $p\in [1,1+4/(n-2))$ and 
$\lam=\lam(x)$ satisfies conditions (i)-(iii) in Hypothesis \ref{mu-hypo}, namely, 
$\lam$ is radial, non-increasing and $\lam_{min}\le\lam(x)\le \lam_{max}$ for all $x$, where 
$\lam_{min}\ge 0$, $\lam_{max}>0$ are the infimum and supremum of $\lam$. 
In  physics, 
an inhomogeneous nonlinearity of a NLS 
  stands for either a correction to the nonlinear power-law response, or  some inhomogeneity in the medium \cite{Berge00,deBFu05}. 
For $p=1+4/n$, the minimal mass blowup problem for inhomogeneous NLS with a potential $V$ 
and $\Om=0$ was  considered in \cite{BaCaDuy10}. 
In the setting of a rotational NLS, we investigate the minimal mass problem for  (\ref{e:psi-VomLz}) and obtain 
 the following  analogue of Theorem \ref{t:gwp-blowup-Om}. 
\begin{theorem}\label{sharpQ:inhom} 
Let $p=1+\frac{4}{n}$ and $(\Om,\ga)$ be any pair in $\R\times \R_+$. Let $\lam$ satisfy Hypothesis \ref{mu-hypo}.
Suppose the initial data $ u_0 \in \Sigma$.   
Then the following holds true for (\ref{e:psi-VomLz}). 
\begin{enumerate}
\item[(a)] If $ \norm{u_0}_{2} <c_0:=\norm{ Q_{\lambda_{max}} }_2$, then the solution $u\in C(\R,\Sigma)$ exists globally in time.  
\item[(b)]  The number $c_0$ in (a) is sharp in the sense that for any $c>\norm{ Q_{\lambda_{max}}}_2$,
there exists $u_0$ in $\Sigma $ satisfying  
  $\norm{u_0}_2 =c$, 
such that  the solution $u\in C(I,\Sigma)$   blows up in  $I=[0,  T_{max})$ for some $T_{max}>0$,  
namely,  $\norm{\nabla u(t)}_2\to \iy$ as $t\to T_{max}$. 
\end{enumerate}\end{theorem}  
Motivated from the line of proof for (\ref{e:nls-H-rot}) presented in Sections 3-5, 
   we first settle the sharp constant $\td{c}_{GN}$ in (\ref{GN:lam(x)}), a Gagliardo-Nirenberg type inequality in an inhomogeneous version,
   by studying the minimization of the $\td{J}$-functional (\ref{J_sn:inhomo}).
   It turns out that the constant $\td{c}_{GN}=\frac{2 n}{n+2}\Vert Q_{\lam_{max}}\Vert_2^{4/n}$
   plays an important role in the proof of 
   Theorem \ref{sharpQ:inhom}. Then, we   prove  a 
 virial identity in Lemma \ref{U:VOm-inhom}. An alternative proof of this lemma could be performed through some method from a hydrodynamical system 
 \cite{Af06,AMS12}.  
Thus, Lemma \ref{U:VOm-inhom} allows us to 
  prove Lemma \ref{J-inhomRNLS},  a criterion of blowup for (\ref{e:psi-VomLz}).   
Here we encounter the difficulty where the solution of the o.d.e. \eqref{J''+J-inhom} for $J(t)$ is not solvable. 
We circumvent this obstacle by writing the unknown $J(t)$ in an implicit integral equation so that we are able to obtain 
some good estimate for $J(t)$ as shown in (\ref{J<Csin2gamma}), which leads to 
establishing  Lemma \ref{J-inhomRNLS}  for the inhomogeneous RNLS.  

We wish to mention that the question of determining the minimal mass blowup solutions 
can be subtle and delicate. In our setting, at first, the ``natural''  threshold seems to be $\norm{\td{Q}}_2$, where $\td{Q}$  
is a solution of 
\begin{align}\label{tQ:inhomo}  
-\frac12\Delta \wtd{Q} - \lambda(x) |\wtd{Q}|^{\frac{4}{n}}\wtd{Q} =-\wtd{Q}\,,
\end{align}
when viewing the identity \eqref{eE:lam(x)} and the condition $\td{E}_{0,0}(u_0)\le 0$ that is required in Lemma \ref{J-inhomRNLS}.
However, a minimizing sequence argument for the $\td{J}$-functional as given in Lemma \ref{j_lam_max} 
and an application of Lemma \ref{J-inhomRNLS} enable us to determine the sharp threshold
in Theorem \ref{sharpQ:inhom}, where we notice the mass concentration at the maximum point  for $\lam(x)$ when
constructing a blowup solution, 
see  Remarks \ref{re:Q} and \ref{rk:gap-Qmax-min}. 
One application of Theorem \ref{sharpQ:inhom} for equations of the inhomogeneous type (\ref{e:psi-VomLz}) is the interesting connections with 
NLS on a manifold, see Remark \ref{r:geometry} for more details. 

In the third component of this paper, in Sections \ref{s:grst-exist-Energy} to \ref{s:orb-stab-u}, we consider the stability of {\em ground state solutions}  for (\ref{e:nls-H-rot}). 
A {\em ground state solution} (g.s.s.) at mass level $c>0$ is a minimizer of the energy functional with constant mass constraint that is defined 
by  
Definition \ref{def:Ic-Q_OmV}. 
For the usual focusing NLS
\begin{align}\label{E:nls-Up}
	i u_t = -\frac12\Delta u  - \lam |u|^{p-1} u ,  \quad   u(0) = u_0 \in {H}^1,
	\end{align}
	 the ground state $Q$ is unique and orbitally stable if $p<1+4/n$, and unstable if $p\ge 1+4/n$. 
In Section \ref{s:grst-exist-Energy},  in the slow rotating setting $|\Om|<\ga$,
we show the existence of  g.s.s. for the RNLS (\ref{e:nls-G-Omga}) with more general class of  nonlinearities. 
Our construction of the g.s.s relies on the concentration compactness method in \cite{CazLion82} and \cite{CarHa15,HiStu04}. 
Then in Section  \ref{s:orb-stab-u},  we prove the orbital stability  for  \eqref{e:nls-G-Omga}  via standard argument. 
  Theorem \ref{t:orb-stab-Zc} shows that  $Z_c$, the set of 
g.s.s. for \eqref{e:nls-G-Omga} is orbitally stable if either $p<1+4/n$, 
or  $p=1+4/n$ and the mass level is below that of $Q_{\lam,1}$ for some optimal constant $\lam$
in the focusing case. 
 \begin{theorem}\label{t:stabi-masscr3} Let $p=1+\frac4n$, $\lam>0$ and $|\Om|<\ga$.   
Let $Q$ be the unique ground state of \eqref{e:grstQ-masscri}. 
Suppose $c<{\norm{Q}_2}$. Then 
 the set of minimizers $Z_c$ is orbitally stable for (\ref{e:nls-H-rot}). 
\end{theorem}

 This result  is a special case of  part (b) in Theorem \ref{t:orb-stab-Zc}, which 
 suggests  that the ``minimal mass'' for the blowup is the same threshold for orbital stability problem under the condition 
 $|\Om|<\ga$.   
When $p=1+4/n$, 
  the energy functional $E_{\Om,\ga}(u)$ for (\ref{e:nls-H-rot}) is unbounded from below on the mass level set $S_{c_0}$ with 
 $ c_0=\norm{Q}_2$, thus, an absolute minimum of the energy does not exist on $S_{c_0}$, see \cite{GuoLY19,GuoSeir14}.
  This also indicates  the existence of the threshold we have obtained in Theorem \ref{t:gwp-blowup-Om} and Theorem \ref{t:stabi-masscr3}. 
Concerning the case $p<1+4/n$,  
 the orbital stability  was proven in   \cite{CazE88,EL89}  
if $|\Om|=\ga$, and recently in \cite{ANenS18} if $|\Om|<\ga$, $n=2, 3$.

Our method allows us to prove the mass-critical case by applying the diamagnetic inequality 
	(\ref{deU<grad-iAu}) along with the $\Sigma$-norm equivalence (\ref{(3.7)}), which are the main novel technical ingredients in the proof.  Moreover, using these techniques and the construction method in \cite{HiStu04}
	we show the existence of stable ground states in 
 Theorem \ref{th:grstQ-masscri} and Theorem \ref{t:orb-stab-Zc} for RNLS (\ref{e:nls-G-Omga}) with a general class of nonlinearities that include  combined inhomogeneous  power terms.  
 The  frequency condition $|\Om|\le \ga$ is  critical  to guarantee stability. 
Physically, if  the angular velocity of rotation exceeds  
 the trapping frequency,  that is, $ |\Om|>\ga$,
 then   
  $V$ cannot provide the necessary centripetal force (that counteract the centrifugal force caused by the rotation), 
and the gas may fly apart.  The counterexample we show in  Section \ref{s:orb-stab-u} agrees with physics observation. 
For attractive particle interactions, 
  it 
 may lead to centrifugal forces destabilizing all rotating states. 
 
\section{Preliminaries}\label{s:prelimin} 
  One can write 
   \eqref{e:nls-H-rot} in the  form of mNLS
\begin{align}
&i u_t =-\frac12(\nabla-i A)^2u  + V_{e}u -\lam |u|^{p-1} u,\label{e:ut-L-p_AV}
\end{align}
where $V_e$ 
denotes the effective electric potential for the Hamiltonian operator
$H_{A,V}=-\frac12(\nabla-i A)^2  + V_e$. 
In particular, 
if $V=\frac{\ga^2}{2}|x|^2$,  
\begin{align}\label{A:Om-x}
&A=\Om\la  -x_2 ,x_1,0\dots,0\ra,
\end{align} 
then 
$V_{e}=
\frac12(\ga^2-\Om^2)(x_1^2+x_2^2)+\frac{\ga^2}{2}x_3^2+\cdots+\frac{\ga^2}{2}x_n^2$  
and 
$iA\cdot\nabla=-\Om L_z=L_\Om$. 
Define
\begin{align}
H_{\Omega, V}:= -\frac12\Delta + V - \Omega  L_z \,. \label{HomV}
\end{align}
 Then  
 $H_{\Om,V}=H_{A,V}$ is essentially self-adjoint in $L^2(\R^n)$. 
In $\R^3$, the operator $L_\Om$ generates a rotation in the sense that: If $(r,\theta,z)$ is the cylindrical coordinate, then 
 $e^{-itL_\Om}f(r,\theta,z)= f(r, \theta+t\Om ,z)$.

\subsection{Wellposedness for RNLS} 
Let us review some results on l.w.p and g.w.p. for the Cauchy theory for 
\eqref{e:nls-H-rot} in the energy subcritical regime.   
   For  $1\le p<1+4/(n-2)$, i.e.,  the $\sH^1$-subcritical case, the local existence and uniqueness  for \eqref{e:nls-H-rot} follow from those of the magnetic NLS \eqref{e:ut-L-p_AV},   based on the fundamental solution constructed in \cite{Ya91}, see \cite{deB91,Mi08} and \cite{AMS12}. 
   The $\sH^s$-subcritical result for \eqref{e:ut-L-p_AV} was considered in \cite{Z12a}  for $1\le p<1+4/(n-2s)$ if $V_e(x)$ 
  is subquadratic and bounded from below. 
  The local wellposedness  for $V_e(x)=-\sum_j \td{\ga}_j^2 x_j^2$
 follows from a Strichartz estimate in Lemma \ref{l:disp-Stri-L}, 
 see Proposition \ref{pr:gwp-lwpUav}.
In the mass-subcritical case $p<1+4/n$ with any data in $L^2$ and in the mass-critical case $p=1+4/n$ with small data in $L^2$,   
the global in time solution exists and is unique \cite{GaZ13a, Z12a}. 

\setcounter{tocdepth}{2} 
\setcounter{subsection}{0}

\begin{proposition}\label{pr:gwp-lwpUav}
Let $1\le p< 1+\frac{4}{n-2}$. Let $u_0\in \sH^1$, $r=p+1$ and  $q=\frac{2(p+1)}{(p-1)}$. 
\begin{enumerate}
\item[(a)] (local existence) There exists  a maximal time interval $I=(-T_{min},T_{max})$, 
 $T_{max}$,  $T_{min}>0$ 
such that \eqref{e:nls-H-rot} has a unique solution $u\in C( I,\sH^1)\cap L^q( I, \sH^{1,r})$.  
\item[(b)] (global existence)  There exists  a unique, $\sH^1$-bounded  global solution in 
$C(\R,\sH^1)\cap L^q_{loc}(\R,\sH^{1,r})$ if one of the following conditions is satisfied: 
\begin{enumerate}
\item[(i)] $1<p<1+\frac{4}{n-2}$\,, $\lam<0$ (defocusing), 
\item[(ii)]  $1\le p<  1+\frac{4}{n}$\,, $\lam>0$ (focusing), 
\item[(iii)] $1+\frac{4}{n}\le p<1+\frac{4}{n-2}$\,, $\lam>0$ (focusing) and  $\Vert u_0\Vert_{\sH^1}< \veps$ for some $\veps=\veps(\lam,n)$ sufficiently small.  
\end{enumerate}
\item[(c)]  If $T_{max}$ (respectively $T_{min}$) is finite, then $\norm{\nabla u(t)}_2\to \iy$ as $t\to T_{max}$
(respectively $ -T_{min}$). 
\item[(d)] On the lifespan interval $(-T_{min},T_{max})$,  the following quantities are conserved in time: 
\begin{align}  
&(mass)\quad M(u)=\int |u|^2 \label{e:mass-OmV}  \quad  \\
&(energy)\quad E_{\Om,V}(u)= \int \left(\frac12 |\nabla u|^2+V|u|^2-\frac{2\lam}{p+1}|u|^{p+1}\right)
+\la L_\Om u, u \ra\label{e:Eu-Vom}
\end{align} 
\item[(e)]\label{i:Lom-conserv} The angular momentum $\ell_\Om(u):=\la L_\Om u,u\ra $ is real-valued and 
\begin{equation} 
\la L_\Om u,u\ra= -\Om\int \bar{u} L_z u \,, \label{e:conservLom-momen}
\end{equation}
\end{enumerate}
where the inner product is defined as 
$\displaystyle	\langle f, g \rangle:= \int f \overline{g}$. 
\end{proposition}
The continuity and conservation laws for the solution map $u_0\mapsto u$ in $\sH^1$ follow from a standard argument using the Duhamel formula \cite{
deB91}. 
The above results extend to l.w.p. and g.w.p. in $\sH^k$ for \eqref{e:nls-H-rot} through a similar proof in view of 
Lemma \ref{l:disp-Stri-L}. 

\subsection{Strichartz estimates for $e^{-itH_{\Om,V}}$} 
The local in time result for \eqref{e:ut-L-p_AV} requires dispersive and Strichartz estimates for the time-dependent propagator $U(t)=e^{-itH_{A,V}}$.
Yajima \cite{Ya91} combines the oscillatory integral operators, bicharacteristics and  integral equation method 
developed in Fujiwara  and Kitada et al's work  \cite{Fuji79,KitaKu81} 
to obtain  
\begin{align*}
 &U(t)f(x)
 =(2\pi it )^{-n/2}\int e^{iS(t,x,y)}a(t,x,y)f(y)dy\, ,\end{align*} 
and it follows that 
\begin{align} 
 & |U(t,x,y)|\le \frac{c_n}{|t|^{n/2}}\,,\qquad \forall \ 0<|t|<\de \label{e:U-ker-timedec} 
 \end{align} 
for some positive constant $\de$,
where $U(t,x,y)$ represents the kernel of $U(t)$ and $S(t,x,y)$ and $a(t,x,y)$ are smooth bounded functions in $((-\de,\de)\setminus\{0\}) \times \R^{2n}$. 
Thus the dispersive estimate holds for any $0<|t|<\de$, 
	\begin{align}
	\Vert U(t)f \Vert_{L^\infty} \lesssim \frac{1}{|t|^{n/2}} \Vert f \Vert_{L^1}\, ,\label{e:L1-infty-disp}
	\end{align} 
which leads to Strichartz estimates \eqref{e:U(t)f-qr-H1} and \eqref{e:stri-Hsqr-L}, 
and hence the local existence on ($-\td{\de},\td{\de}$) by standard arguments,  
 see  \cite{deB91,KTao98,Mi08,Z12a}. 

\begin{definition}\label{de:adm.pair-qr} We call $(q,r)$ an admissible pair if $q,r\in [2,\iy]$ satisfy  $(q,r,n)\neq (2,\iy,2)$ and
\begin{align*} 
\frac{2}{q}+\frac{n}{r}=\frac{n}{2}.
\end{align*}   
\end{definition}

By Duhamel formula, $u$ is a  weak solution of  
\eqref{e:nls-H-rot} 
is equivalent to 
 \[ u=U(t)u_0+i \int_0^t U(t-s) \lam |u|^{p-1} u ds.
\]
From \eqref{e:L1-infty-disp} and \cite[Lemma 3.1]{Ya91} 
we  have  the following Strichartz estimates in 
 weighted Sobolev spaces.
\begin{lemma}\label{l:disp-Stri-L} Let $I=[-T,T]$, $T<\de$ small. 
Let $(q,r)$ and $(\td{q},\td{r})$ be any admissible pairs. Then
\begin{align}
& \Vert  U(t)f \Vert_{L^q(I,L^r) } \le C \Vert  f\Vert_{L^2(\mathbb{R}^n)} ,\notag\\
&\Vert \int_0^t  U(t-s)F ds\Vert_{L^q( I,L^r)}\le C_{n,q,\td{q}}
  \Vert F \Vert_{L^{\td{q}'}(I, L^{\td{r}'})} ,\notag\\
& \Vert  U(t)f \Vert_{L^q(I,\sH^{1,r}) } \le  C\Vert  f\Vert_{ \sH^1},\label{e:U(t)f-qr-H1}\\
&\Vert \int_0^t  U(t-s)F ds \Vert_{L^q(I, \sH^{1,r})}
\le C_{n,q,\td{q}}
  \Vert F \Vert_{L^{\td{q}'}(I, \sH^{1,\td{r}'})},\label{e:stri-Hsqr-L}
\end{align} 
where  
\begin{align*}
& \Vert u \Vert_{L^q(I, L^r)}= \left(\int_I \big(\int |u(t,x)|^r dx\big)^{q/r} dt \right)^{1/q},\\
& \Vert u \Vert_{L^q(I, \sH^{1,r})}= \left(\int_I \norm{ u(t,\cdot)}_{\sH^{1,r}}^q dt \right)^{1/q},\\
&\norm{u}_{\sH^{1,r}}:=\norm{ \nabla u}_{L^r}+\norm{ x u}_{L^r}+\norm{u}_{L^r}.
\end{align*} 
\end{lemma}
The lemma here applies to the case where $V_e$ is subquadratic, e.g. $V_e(x)=\pm\sum_j \td{\ga}_j^2x_i^2$ for 
$\td{\ga}_j\ge 0$. 
This is a generalized version for the Strichartz estimates given in  \cite{deB91,Z12a}  where  $V_e$ is required being bounded from below.  
In the proof of the lemma, we directly study the action of $U(t-s)$ on the space $\sH^{1,r}$ 
based on \cite[Lemma 3.1]{Ya91}, 
an oscillatory integral operator result of  Yajima. 
This provides a treatment for $A$ sublinear and $V$ subquadratic in the time-dependent case that covers those in 
Proposition \ref{pr:gwp-lwpUav}. 
Such treatment is more general and also very different than the commutator method used in 
 \cite{AMS12,Car11time}.

\section{Virial identity for NLS with rotation} 
\subsection{Virial identity for (\ref{e:nls-H-rot})} 
In this subsection we derive the virial identity associated with  equation \eqref{e:nls-H-rot}. 
Let $A$ be given as in (\ref{A:Om-x}) 
and 
$L_\Om=
 i A\cdot \nabla$.  
The proof of Theorem \ref{t:gwp-blowup-Om} is based on the following  lemma for the variance 
\begin{equation}
	J(t) := \int |x|^2 |u|^2.\label{e:variance-J(t)}
	\end{equation}

\setcounter{tocdepth}{2} 
\setcounter{subsection}{0}
 
\begin{lemma}[Virial identity]\label{l:virial-J(t)-Vom} Let $u$ be solution of \eqref{e:nls-H-rot} with initial data 
$u_0\in  \sH^1$.  Then 
\begin{align}
J'(t) =& 2 \Im \int x \overline{u} \cdot \nabla u 
 \label{eJ':xu-deU_A}\\
J''(t)=& 2 \int |\nabla u|^2-2 \gamma^2 \int |x|^2 |u|^2
	-  2n\lam  \frac{p-1}{p+1} \int |u|^{p+1}\qquad \notag\\
	=&4E_{\Om,V}(u)-4\ga^2\int |x|^2 |u|^2
+\frac{2\lam}{p+1}(4-n(p-1)) \int |u|^{p+1}-4\la L_\Om(u),u\ra. \label{eJ'':4E(u)OmV}
\end{align}
\end{lemma}

 Note that  $\la  L_\Om u, u\ra$ is real since $L_\Om$ is selfadjoint. 
The virial  inequality will be used to analyze wave collapse in finite time. 
Identity (\ref{eJ'':4E(u)OmV})  can also be derived from the  magnetic analog \cite{FV09,Gar12,BHIM}. 
In Section \ref{s:Virial-inhom}, we will give a version of the virial identity in the case of  inhomogeneous nonlinearity. 

\begin{proof}[Proof of Lemma \ref{l:virial-J(t)-Vom}]    
We may assume $u\in C^1(I, C^2_0\cap \sH^1)$ and $I=[0,T_{max})$.
For general data 
the identities \eqref{eJ':xu-deU_A} and \eqref{eJ'':4E(u)OmV}   follow from a standard approximation argument. 
First we obtain the identity \eqref{eJ':xu-deU_A} 
from \eqref{e:variance-J(t)} and \eqref{e:nls-H-rot}  by integration by parts, 	
where note that 
\begin{enumerate}
\item[(i)] $V$ and $\lam$ are real-valued; 
  \item[(ii)] $x\cdot A=0$;
 \item[(iii)] $L_{z}= i \left(x_2 \partial_{x_1} - x_1 \partial_{x_2} \right)$ is self-adjoint 
and $L_{z} (|x|^2)=0$. 
\end{enumerate}
Then, it follows by differentiating in $t$ on (\ref{eJ':xu-deU_A}) that
	\begin{equation*}
	\begin{aligned}
	J''(t)
	&= 2 \Im \left( \int x \overline{u} \cdot \nabla u_t
	+ \int x \overline{u}_t \cdot \nabla u \right)
	= 2 \Im \left( - \int \nabla \cdot ( x \overline{u} ) u_t + \int x \overline{u}_t \cdot \nabla u \right) \\
	& = -2n \Im \left( \int \overline{u} u_t \right)
	-4 \Im \left( \int x \cdot \nabla \bar{u} \,u_t \right) \\
	& := -2n S -4 T.
	\end{aligned}
	\end{equation*} 
Noting that  $\displaystyle \int \overline{u} L_{z} u$ is real,  
we have by a simple calculation 
	\begin{equation*}
	S = - \frac{1}{2} \int |\nabla u|^2
	- \frac{\gamma^2}{2} \int |x|^2 |u|^2
	+ \lambda \int |u|^{p+1}
	+ \Om\int \overline{u}  L_z u.
	\end{equation*} 
To compute $T$,
one has
	\begin{equation*}
	\begin{aligned}
	T=& 
	  \Im \left(\frac{i}{2} \int x \cdot \nabla \overline{u} \, \Delta u \right)
	+ \Im \left( -i\frac{\gamma^2}{2}\int x \cdot \nabla \overline{u}\,(  |x|^2 u ) \right) \\
	&+  \Im \left(i\lam \int x \cdot \nabla \overline{u}	\, |u|^{p-1} u \right)
	+ \Im \left(i \Om \int x \cdot \nabla \overline{u} \, L_z u \right) \\
	:=& T_1 + T_2 + T_3 + T_4.
	\end{aligned} 
	\end{equation*}
For $T_1$, integration by parts gives
	\begin{align*}
	T_1
	=& \Im \left( -i \frac{1}{2} \sum_{j,k=1}^n
	\int \left( \delta_{j,k} \overline{u}_{x_j} + x_j \overline{u}_{x_j x_k} \right) u_{x_k} \right) \\
	=& \Im \left( -i \frac{1}{2} \int |\nabla u|^2 \right)
	+ \Im \left( - i \frac{1}{2} \sum_{j, k=1}^n \int x_j \overline{u}_{x_j x_k} u_{x_k} \right)
	:= -\frac{1}{2} \int |\nabla u|^2 + T_{1, 1}\, ,  \quad 
	\end{align*}
where $\de_{j,k}$ denote the Kronecker delta. 
To compute $T_{1,1}$, one has
	\begin{align*}
	T_{1,1}
	& = \Im \left( \frac{i}{2} \sum_{j, k=1}^n \int \left( x_j u_{x_k} \right)_{x_j} \overline{u}_{x_k} \right)
	= \Im \left( i \frac{1}{2} \sum_{j, k=1}^n \int \left( u_{x_k} + x_j u_{x_k x_j} \right) \overline{u}_{x_k} \right) \\
	& = \frac{n}{2} \int |\nabla u|^2
	+ \Im \left( \frac{i}{2} \sum_{j, k=1}^n \int x_j u_{x_k x_j} \overline{u}_{x_k} \right)
	= \frac{n}{2} \int |\nabla u|^2 {-} T_{1,1}.
	\end{align*}
This shows that $\displaystyle T_{1,1} = \frac{n}{4} \int |\nabla u|^2$,
and so $T_{1} = \frac{n-2}{4} \int |\nabla u|^2$.
	For $T_2$ and $T_3$, applying divergence theorem shows that
	\begin{align*}
	T_2 =& \frac{n+2}{4} \ga^2\int |x|^2 |u|^2.\\  
	T_3 =& \frac{- \lam n}{p+1} \int |u|^{p+1}.
	\end{align*}
For $T_4$,  integration by parts gives 
	\begin{align*}
	T_{4}=&  
	  \Im \left( -\Om \int \sum_{k=1}^n x_k \overline{u}_{x_k} \left( x_2 u_{x_1} - x_1 u_{x_2} \right) \right) \\
	 =& \Im \left( \Om \int \left( \sum_{k=1}^n x_k \overline{u}_{x_k} x_2 \right)_{x_1} u \right)
	- \Im \left( \Om \int \left( \sum_{k=1}^n x_k \overline{u}_{x_k} x_1 \right)_{x_2} u \right) \\
	 =& \Im \left( i \Om \int u \overline{L_{z} u} \right)
	+ \sum_{k=1}^n  \Im \left( - \Om \int \left( x_k x_2 u \right)_{x_k} \overline{u}_{x_1} \right)
	- \sum_{k=1}^n  \Im \left( -\Om \int \left( x_k x_1 u \right)_{x_k} \overline{u}_{x_2} \right) \\
	 =& \Om \int u \overline{L_{z} u}
	+ \Im \left( -n \Om \int u \left( x_2 \overline{u}_{x_1} - x_1 \overline{u}_{x_2} \right) \right)
	+ \Im \left( - \Om \int u \left( x_2 \overline{u}_{x_1} - x_1 \overline{u}_{x_2} \right) \right) \\
	+&\Im\left( - \Om\int x\cdot\nabla u \left( x_2 \bar{u}_{x_1}-x_1\bar{u}_{x_2}\right) \right)\\
	 =& -n \Om \int u \overline{L_{z} u} {- T_{4}}. \qquad 
	\end{align*}
This shows that $\displaystyle T_{4} = -\frac{n}{2} \Om \int u \overline{L_{z} u}$. 
Collecting all $T_i$ terms yields
	\begin{equation*}
	\begin{aligned}
	T
	= \frac{n-2}{4} \int |\nabla u|^2
	+ \frac{(n+2) \ga^2}{4} \int |x|^2 |u|^2
	- \frac{n \lam}{p+1} \int |u|^{p+1}
	-\frac{n}{2} \Om\int u \overline{ L_z u}.
	\end{aligned}
	\end{equation*}
Finally, we obtain the virial identity (\ref{eJ'':4E(u)OmV})
	\begin{equation*}
	\begin{aligned}
	J''(t)
	& = -2n \left( - \frac{1}{2} \int |\nabla u|^2
	- \frac{\ga^2}{2} \int |x|^2 |u|^2
	+ \lam \int |u|^{p+1}
	+ \Om\int \overline{u}  L_z u \right) \\
	& \quad -4 \left( \frac{n-2}{4} \int |\nabla u|^2
	+ \frac{(n+2) \gamma^2}{4} \int |x|^2 |u|^2
	- \frac{n \lambda}{p+1} \int |u|^{p+1}
	-\frac{n}{2} \Om\int u \overline{L_z u} \right) \\
	& = 2 \int |\nabla u|^2
	-2 \gamma^2 \int |x|^2 |u|^2
	- 2n\lam \left(\frac{p-1}{p+1} \right) \int |u|^{p+1}. 
	\end{aligned}
	\end{equation*} 
\end{proof}

\setcounter{tocdepth}{2} 
\setcounter{subsection}{1}

\subsection{The conservation of angular momentum}\label{ss:conserv-angularOm} 
For $A$ given in (\ref{A:Om-x}) 
we show (\ref{e:conservLom-momen}) in Proposition \ref{pr:gwp-lwpUav} 
that is, the angular momentum $\ell_\Om(u)$ 
  is conserved  in $t$. 

\begin{proof} 
 Recall that $ \ell_\Om (u) = -\Om\int \bar{u} L_zu$. It suffices to show $\frac{d}{dt} \int \bar{u} L_z u =0$.
Differentiating  $\la L_zu, u\ra$ with respect to $t$ 
and substituting  $u_t$ from equation \eqref{e:nls-H-rot},
we have 
\begin{align*}
 \frac{d}{dt} \left( \la  L_z u,u \ra\right) 
=&2\Re\int \bar{u}_t L_zu \\
 =&\Re\int\De \bar{u} (x_2\pa_{x_1}{u}-x_1\pa_{x_2}{u})-2\Re\int V\bar{u} (x_2\pa_{x_1}-x_1\pa_{x_2}){u}\\ 
 +&2\Re\int \lam \bar{u}|u|^{p-1} (x_2\pa_{x_1}-x_1\partial_{x_2}) u\\
 =:& 2\Re (\frac12I_1 - I_2 + I_3) .
 \end{align*}
 Noting $[\De,L_z]=0$, 
we obtain
 $I_1= - \overline{I_1}$, that is, $\Re I_1=0$. 
For $ I_2 $, we have 
\begin{align*}
 I_2=& \int V\bar{u} (x_2\partial_{x_1}-x_1\partial_{x_2}) u\\
=& -\int (\widetilde{L}_z V) |u|^2 - \int V {u} \widetilde{L}_z\bar{u}\\
=&0-\int V u \widetilde{L}_z\bar{u} \, ,\end{align*}
which implies  $\Re I_2= 0$,
where we have defined 
$\widetilde{L}_z=x_2\pa_{x_1}-x_1\pa_{x_2}$
and noted 
$ \widetilde{L}_z V(x)=0$.

For $ I_3 $, since $\wtd{L}_z$ is skew-symmetric, we have 
\begin{align*}
 I_3=&\lam\int  \bar{u}|u|^{p-1} \wtd{L}_z  u\\
 =&-\lam\int u |u|^{p-1}  (x_2\pa_{x_1}- x_1\pa_{x_2})\bar{u} -\lam\int |u|^2  (x_2\pa_{x_1}- x_1\pa_{x_2})(| u|^{p-1})\\
=&-\lam\int u |u|^{p-1} \overline{\wtd{L}_zu} -\lam(p-1) \Re \int (| u|^{p-1}u) ( x_2\pa_{x_1} - x_1\pa_{x_2}) \bar{u}.
\end{align*}
Taking the real part gives
\begin{align*}
 &\Re (I_3)= -p \Re(I_3),
  \end{align*}
that is, $ \Re(I_3) =0$.
Combining the results of $ I_1, I_2, I_3 $ above, we obtain
\begin{equation*}
 \frac{d}{dt} ( \ell_\Om(u)) = 0 .
\end{equation*}
\end{proof}

\section{Blowup criterion for $\Om\ne 0$}\label{ss:proof-blowup-rotNLS} 
In this section, we prove some general criteria on finite time blowup for the solutions of  RNLS (\ref{e:nls-H-rot})
based on the virial identity and conservation of mass, energy and angular momentum.
\begin{lemma}\label{l:J-J'-blup} Let  $p\in [1+\frac{4}{n}, 1+\frac{4}{n-2})$ and $\lam>0$. 
Suppose $0\ne u_0 \in \sH^1$ satisfies one of the following conditions: 
\begin{enumerate}
\item[(a)] 
 $ E_{\Om,V} (u_0 )-\ell_\Om (u_0)\le \frac{\ga^2}{2}J(0)$, if  $p=1+4/n$. 
\item[(b)] $\frac{\ga^2}{2}J(0)<E_{\Om,V} (u_0 )-\ell_\Om (u_0)\le  \frac{\ga}{2} |J'(0)|$, if $p= 1+4/n$.
\item[(c)]  $E_{\Om,V} (u_0 )-\ell_\Om (u_0)<0$, if $p> 1+4/n$.  
\item[(d)] $E_{\Om,V} (u_0 )-\ell_\Om (u_0)=0$ and $J'(0)<0$, if $p> 1+4/n$. 
\end{enumerate}
Then there exists $0<T_*<\iy$ such that 
the corresponding solution  $u$ of Equation \eqref{e:nls-H-rot}  blows up on $[0,T_*)$ 
satisfying  \begin{align}\label{deU:T-t}
&\Vert \nabla u\Vert_2 \to\iy
\quad as \; t\to T_*\,.
\end{align}
\end{lemma}
  
\begin{remark} The condition in part (a) of the lemma is equivalent to 
\begin{align}
&E_{0,0} (u_0 )=\frac12\int |\nabla u_0|^2 -\frac{\lambda n}{n+2}\int |u_0|^{2+\frac{4}{n}}\le 0 .\label{e:E0<Lom}
\end{align}
\end{remark}

\begin{remark}[Examples of $u_0$ verifying conditions in Lemma \ref{l:J-J'-blup}]
For the conditions (a) and (c)-(d) in the lemma, there exist initial data that are radially symmetric
as given in \eqref{Q:Calpha} and \eqref{E:u0-enux2Q} in
the proofs for Theorem \ref{t:gwp-blowup-Om} and Proposition \ref{p:blowup-masssup},  respectively.

Here we provide some other type of examples that are not radially symmetric. 
\begin{enumerate}
\item[(i)] {\em Conditions (a) and (c)}. 
Let $x=(r,\theta,z)$,  where  $(r,\theta)$ are the polar coordinates of $(x_1,x_2)\in\R^2$ and $z=(x_3,\dots,x_n)\in\R^{n-2}$. 
Let $\vphi$  be real-valued and belong in $\Sigma \cap \{ \int_{\R^n} r^{-2}|\vphi(x)|^2 dx<\iy  \}$. 
Consider the initial value
	\begin{align*}
	u_0(x)=Ce^{im\theta}\vphi(x),
	\end{align*}
where $C$, $m$ are real constants.  
We have
	\begin{align*}
	&E_{\Omega,V}(u_0)-\ell_\Omega(u_0)\\
	=&\frac{C^2m^2}{2} \int \frac{|\vphi|^2}{r^2}
	+\frac{C^2}{2} \int |\nabla \vphi|^2
	+\frac{\gamma^2C^2}{2} \int |x|^2|\vphi|^2
	-\frac{2\lambda |C|^{p+1}}{p+1} \int |\vphi|^{p+1}.
	\end{align*}
Then the conditions in  Lemma 4.1 (a) and (c) can be  achieved if we choose $C>0$ sufficiently large.
 

\item[(ii)] {\em Conditions (b) and (d).} Consider the initial value
	\begin{align*}
	u_0(x)=Ce^{i\omega\cdot x}Q(x-x_0),
	\end{align*}
where $C\in\R$, 
$\omega\in \R^n$, $0\ne x_0\in\mathbb{R}^n$, 
and $Q$ is the positive radial solution of (\ref{e:grstQ-masscri}). 
For condition (b), 
we will choose appropriate $C$, $\omega$,  and $x_0$ such that the following is satisfied: For $p=1+\frac{4}{n}$\,,
	\begin{align}\label{E:J(0)<dJ(0)}
	\frac{\gamma^2}{2}J(0)<E_{\Omega,V}(u_0)-\ell_\Omega(u_0)\leq \frac{\gamma}{2}|J'(0)| \,.
		\end{align}
 We compute using \eqref{eJ':xu-deU_A} and
 Lemma \ref{l:pohozaev-Qp} to obtain \n
  \mbox{$J'(0)=2 C^2 \omega\cdot x_0 \int  Q^2$} 
and \begin{align*}
	&E_{\Omega,V}(u_0)-\ell_\Omega(u_0)\\
	=&\int\left( \frac{C^2|\omega|^2}{2}Q^2
	+\frac{C^2}{2}|\nabla Q|^2
	+\frac{\gamma^2}{2}|x|^2C^2Q^2(x-x_0)
	-\frac{2\lambda |C|^{p+1}}{p+1}Q^{p+1}\right) \\
	=&\frac{C^2|\omega|^2+nC^2-n |C|^{2+\frac{4}{n}}+\gamma^2 C^2|x_0|^2}{2} \int Q^2
	+\frac{\gamma^2 C^2}{2}\int |x|^2 Q^2\,.
	\end{align*}
It is easy to observe that  the first inequality in (\ref{E:J(0)<dJ(0)}) is valid if
	\begin{align}\label{eq:lemma-4.1-b-1}
	0<C<\left( \frac{|\omega|^2+n}{n} \right)^{n/4},
	\end{align}
and the second inequality in \eqref{E:J(0)<dJ(0)} is valid if
	\begin{align}\label{eq:b2}
	|\omega|^2+n-nC^{\frac{4}{n}}+\gamma^2 |x_0|^2
	\leq \gamma |\omega\cdot x_0| \qquad
	\text{and} \qquad
	\frac{\gamma \int |x|^2 Q^2}{\int  Q^2}
	\leq |\omega\cdot x_0|.
	\end{align}
But both the conditions  \eqref{eq:lemma-4.1-b-1} and \eqref{eq:b2} hold simultaneously
by choosing $C$, $\omega$ and $x_0$ such that $|\omega|^2+n-nC^{\frac{4}{n}}$ is a small positive number and $|\omega\cdot x_0|$  large. 

For condition (d), 
we will choose appropriate $C$, $\omega$, and $x_0$ such that the following is satisfied: For $p>1+\frac{4}{n}$,
	\begin{align}\label{EOmV:0:deJ}
	E_{\Omega,V}(u_0)-\ell_\Omega(u_0)=0 \qquad 
	\text{and} \qquad
	J'(0)<0 \,.
		\end{align}
 By the previous calculations using the  Pohozaev identities in Lemma \ref{l:pohozaev-Qp},
 the above equality and inequality in (\ref{EOmV:0:deJ}) are equivalent to
	\begin{align*}
	|C|^{p-1}
	=& \frac{\big(2(p+1)-n(p-1)\big) (|\omega|^2+\ga^2|x_0|^2)+2n(p-1)}{8} \\
	+&\frac{\big(2(p+1)-n(p-1)\big)\gamma^2 \int |x|^2Q^2}{8\int Q^2}
	\end{align*}
and
	\begin{align*}
	 \omega\cdot x_0 \int  Q^2<0\,,
	\end{align*}
respectively.
We see that one can choose $\omega$ and $x_0$ such that 
\mbox{$\omega\cdot x_0<0$} so the second inequality is satisfied,
and then determine $C$ according to the first equality. 
\end{enumerate}
\end{remark}

We will also need the uncertainty principle given in  \cite{Wein83,Zh00}. 
\begin{lemma}\label{l:UP}
\begin{align}
&\int_{\R^n} |u|^2dx\le \frac{2}{n}\left(\int_{\R^n} |\nabla u|^2dx\right)^{1/2} \left(\int_{\R^n} |x|^2 |u|^2dx\right)^{1/2},\label{e:UP-deu-xu2} 
\end{align}
where the equality is achieved with $u=e^{-|x|^2/2}$.
\end{lemma} 

Now we prove Lemma \ref{l:J-J'-blup}.
\begin{proof}[Proof of Lemma \ref{l:J-J'-blup}]  We divided the proof into three cases.  

\nd{\bf Case 1.}   Let $p=1+4/n$. 
From (\ref{eJ'':4E(u)OmV}) in  Lemma \ref{l:virial-J(t)-Vom}, along with \eqref{e:conservLom-momen} and  \eqref{e:Eu-Vom}  
we have   if  $p=1+4/n$,  
\begin{align}
 J''(t) +4\ga^2 J(t)=&4E_{\Om,V}(u_0)-4 \ell_\Om(u_0).\label{e:J''+J-EL0p}
 \end{align}
Solving this o.d.e. gives
 \begin{align}
J(t) =& (J(0)-\frac{1}{\ga^2}(E_{\Om,V}(u_0)-\ell_\Om(u_0)) ) \cos 2\ga t + \frac{J'(0)}{2\ga} \sin 2\ga t +\frac{1}{\ga^2} 
( E_{\Om,V}(u_0)-\ell_\Om(u_0)) \notag\\
=& C \sin( 2\ga t +\beta)+\frac{1}{\ga^2} (E_{\Om,V}(u_0)- \ell_\Om(u_0)) ,\label{J(t)-Csin}
\end{align} 
where  $C> 0 $ and $ \beta \in (0,\pi) $ are constants given by 
\begin{align*}
&C^2 = \left( J(0) -\frac{1}{\ga^2} (E_{\Om,V}(u_0) -\ell_\Om(u_0)) \right)^2 + \frac{1}{4\ga^2}\left(J'(0)\right)^2,\\ 
&\sin \beta=\frac{J(0)-\frac{1}{\ga^2} (E_{\Om,V}(u_0) -\ell_\Om(u_0))}{C}\ge 0\,,\quad (\text{by condition (a)})\\  
&\cos \beta=\frac{J'(0)}{2\ga C}\,.
\end{align*} 
 
Claim. \begin{align}\label{e:J(t)-0}  
 &C\ge \frac{1}{\ga^2} \left\vert E_{\Om,V}(u_0)- \ell_\Om(u_0)\right\vert.
\end{align}
Indeed,  the condition (a) always implies  (\ref{e:J(t)-0})  if $E(u_0)- \ell_\Om(u_0)$ is either positive  or negative. 
Now, from the equation 
  \begin{equation}\label{e:J(t)<Csin+ELom}
 0 \le J(t)= C\sin (2\ga t + \beta ) + \frac{1}{\ga^2} (E_{\Om,V}(u_0)- \ell_\Om(u_0))
 \end{equation} 
 we see that there must exist $ T_* \in (0, \frac{3\pi}{4\ga}) $ such that 
 \begin{align*}
  \lim_{ t\to T_*} J(t) = 0.  
\end{align*}
Furthermore, since $\lim_{t\to T_*}J'(t)=2\ga C\cos (2\ga T_*+\beta)$ is finite, 
we deduce from (\ref{e:UP-deu-xu2}) that there exists some constant $c_0$ such that 
\begin{align}\label{deU-T-t}
&  \Vert \nabla u(t)\Vert_2\ge  \frac{c_0 \Vert u_0\Vert_2^2}{\sqrt{T_*-t}} \,.
\end{align}

\vs{.20in}
\nd {\bf Case 2.} Let $p=1+4/n$.  If  $\frac{\ga^2}{2}J(0)<E_{\Om,V} (u_0 )-\ell_\Om (u_0)\le  \frac{\ga}{2} |J'(0)|$, then 
(\ref{e:J(t)-0}) still holds. It follows from the same line of proof in Case 1 that there exists 
  some $T_*$ such that  \eqref{deU:T-t} holds as $t\to T_*$.  

\vs{.20in}
\nd 
{\bf Case 3}. Let $u_0$ satisfy either condition (c) or (d). 
 From (\ref{eJ'':4E(u)OmV}) 
we have   if  $p\ge 1+4/n$,  
\begin{align*}
&J''(t)\le 4(E_{\Om,V}(u_0)-\ell_\Om(u_0)) .
\end{align*}
Then integrating twice yields 
  \begin{align*}
&J(t) \leq J(0) + J'(0) t + 2(E_{\Om,V}(u_0)- \ell_\Om(u_0))t^2. 
 \end{align*}
 We see   either condition (c) or (d) suggests  that  there exits $T_*<\infty$ such that 
 \[  J(t) \to 0\quad   
 \text{as $t\to T_*$}\] 
which implies (\ref{deU:T-t}) by (\ref{e:UP-deu-xu2}).  
 This concludes the proof of Lemma \ref{l:J-J'-blup}. 
\end{proof}

\begin{remark}\label{r:E00-T}  Condition (a) in the lemma suggests that if 
$E_{0,0}(u_0)\le 0$, 
then finite time blowup occurs.  
However,  in the absence of a potential and a rotation, i.e., $\Om=\ga=0$, 
the condition $E_{0,0}(u_0)=0$ alone may not imply blowup solutions in the mass-critical case, see \cite[Theorem 4.2]{Wein83}. 
If in addition $J'(0)\le 0$, then a closer look over the proof in Case 1 shows $T_*\in (0, \frac{\pi}{2\ga}]$. 
\end{remark}

\begin{remark}\label{re:T_star}
Here we elaborate on the existence $0<T_* <\frac{3\pi}{4\ga}$ in the proof of Lemma \ref{l:J-J'-blup}
under the conditions (a) and (b) where $p=1+4/n$.  
In fact,  this mainly follows from the observation that the function 
$C\sin (2\ga t + \beta)$ has an amplitude greater or equal to 
$ \frac{1}{\ga^2} | E_{\Om,V} (u_0 )-\ell_\Om (u_0)  | $, in view of (\ref{J(t)-Csin}). 

We divide our  discussions in three elementary cases. 
\begin{enumerate}
\item[(i)] If $0<E_{\Om,V} (u_0 )-\ell_\Om (u_0)\le \frac{\ga^2}{2}J(0)$, 
then from (\ref{J(t)-Csin}), we see $J(\tau)\le 0$   when $\tau=\frac{1}{2\ga}(\frac{3\pi}{2}-\beta)$. 
Since $\beta\in (0,\pi)$, 
 $\tau\in (\frac{\pi}{4\ga}, \frac{3\pi}{4\ga} ) $.  
Since $J(0)>0$, by mean value property, $J(t)$ has a zero 
in $(0, \frac{3\pi}{4\ga})$. 

\item[(ii)] If $E_{\Om,V} (u_0 )-\ell_\Om (u_0)=0 $, then $C\ge J(0)$. 
 We have from (\ref{J(t)-Csin})
\[  J(t)=C\sin (2\ga t+\beta)\,,
\]
where we see there is $\tau\in (0, \frac{\pi}{2\ga})$ s.t. $\sin (2\ga t+\beta)=0$.
 This $\tau$ is the first positive zero of $J(t)$. 

\item[(iii)] If $E_{\Om,V} (u_0 )-\ell_\Om (u_0)<0 $,  then 
  $\sin (2\ga \tau + \beta)=0 $ when  $\tau=\frac{1}{2\ga}(\pi-\beta)$.  
This implies $J(\tau)<0$ for $\tau\in (0, \frac{\pi}{2\ga}) $. 
Therefore, since $J(0)>0$, there exists  some zero of $J(t) $ in  $(0,\frac{\pi}{2\ga}) $ by mean value property. 
\end{enumerate}
Summarizing above, we see that under the condition (a) or (b) in Lemma 4.1, 
there must exist $T_*\in (0,\frac{3\pi}{4\ga})$ such that $J(t)\to 0$ as $t\to T_*\,$.  
\end{remark} 

\section{Blowup for RNLS \eqref{e:nls-H-rot}}\label{s:pf-Thm:VLom} 
\subsection{Proof of Theorem \ref{t:gwp-blowup-Om}}
The proof of part (a) of Theorem \ref{t:gwp-blowup-Om} relies on the sharp Gagliardo-Nirenberg inequality, where the optimal constant is evaluated in terms of  the ground state $Q=Q_{\lam,1}$ as given in (\ref{e:grstQ-masscri}). 
Part (b) of Theorem \ref{t:gwp-blowup-Om} follows as an application of Lemma \ref{l:J-J'-blup}. 
We begin with the following classical lemma on the Pohozaev identity. 

\setcounter{subsection}{0}

\begin{lemma}\label{l:pohozaev-Qp}   %
Let $p\in (1,1+4/(n-2))$. Suppose $u\in H^1$ is a positive solution of the elliptic equation 
\begin{align}
a \De u+ b u^{p}=cu\,, \label{e:u-elliptic_abc}
\end{align} 
where $ac>0$ and $bc>0$. 
Then 
\begin{align}\label{po:deU-p+1}
&2a\int |\nabla u|^2= bn \frac{p-1}{p+1}\int |u|^{p+1}\\
&2c\int | u|^2=b \left(2-\frac{n(p-1)}{p+1}\right) \int |u|^{p+1} \notag\\
& a\left(\frac{2(p+1)}{n(p-1)}-1\right) \int |\nabla u|^2= c\int |u|^2\,. \notag
\end{align}
The solution of (\ref{e:u-elliptic_abc}) will be denoted $u_{\lam,\om}=Q_{\lam,\om}$ when 
  $a=\frac12$, $b=\lam$  and $c=\om$. 
  Equation \eqref{e:u-elliptic_abc}  enjoys the following scaling invariance: 
\begin{enumerate}
\item[(i)] $\om$-scaling: $ u_{\lam,\om}(x)=\om^{1/(p-1)} u_{\lam,1}(\om^{1/2} x)$
\item[(ii)] $\lam$-scaling: $  u_{\lam,\om}(x)=\lam^{-1/p} u_{1,\lam^{-1/p}\om}(\lam^{1/2p} x) $.
\end{enumerate}\end{lemma}  
\begin{proof}  Multiplying 
$x\cdot \nabla u$ and then integrating on both sides of  (\ref{e:u-elliptic_abc}),  we obtain  
\begin{align*}  a(\frac{2}{n}-1 ) \int |\nabla u|^2= c\int |u|^2-\frac{2b}{p+1} \int |u|^{p+1} . 
\end{align*} 
This and the obvious equation 
\begin{align*}
a\int |\nabla u|^2 -b\int |u|^{p+1}=-c\int  |u|^2 
\end{align*}
combine to prove the the three identities 
in the lemma. The statement on the scaling invariance 
follows from the uniqueness of the positive ground state solution $Q_{\lam,\om} $, cf. \cite{Kw89}.
\end{proof}
For $a=\frac12$, $\lam>0$, $c=1$ the   equation (\ref{e:grstQ-masscri}) 
 has a unique positive radially symmetric solution $Q= Q_{\lam,1}$  in $H^1\cap C^2$ 
 which is decreasing  exponentially,  see 
 \cite{Kw89, Wein83}. 
 If $p=1+4/n$,  
 we have $\norm{Q_{\lam,1}}_2^{4/n}  = \lam^{-1}\Vert Q_{1,1}\Vert_{2}^{4/n} $, and by Lemma \ref{l:pohozaev-Qp},
 $Q=Q_{\lam,1}(x) $ satisfies 
\begin{equation}\label{e:pohoz-Qlam}
\int |\nabla Q_{\lam,1}|^2 = \frac{2\lam n}{n+2}  \int | Q_{\lam,1} |^{2+4/n}\, .
\end{equation} 
This is equivalent to 
\begin{align}  E_{0,0}(Q)=0,  \label{lowest-energy_Q}  
\end{align}
where  \begin{align*}
E_{0,0}(u)=\frac12\int |\nabla u|^2-\frac{2\lam}{p+1}\int |u|^{p+1}\,.
\end{align*}

Next we turn to the sharp Gagliardo-Nirenberg inequality, which states that  
there exists an optimal constant $c_{GN}=c(p,n)$ such that for all $u$ in $H^1$  
\begin{align*}
\int |u|^{p+1} \leq c_{GN} \norm{ u}_2^{2-\frac{(p-1)(n-2)}{2}} \norm{\nabla u}_2^{n(p-1)/2} \,.
\end{align*} 
If  $p=1+2\sigma$,  the  constant $c_{GN}\inv$ is given by the minimal of the $J$-functional:
 \begin{align}\label{c_gn-Qp}
& c\inv_{GN}=\min_{0\ne u\in H^1} J_{\sigma,n}(u)\,,
  \end{align} 
 where 
 \begin{align}\label{J_sn[u]}
J_{\sigma,n}(u)=
\frac{ \Vert u\Vert_2^{2+2\sigma-n\sigma } \Vert \nabla u\Vert_2^{n\sigma}}{\Vert u\Vert_{2\sigma+2}^{2\sigma+2}}\,.
\end{align}
Note that $J_{\sigma,n}$ has the scaling invariance $J_{\sigma,n}(\beta u(\al x))=J_{\sigma,n}(u)$ for all $\al,\beta>0$.
Thus  $Q=Q_{\lam,1}$ is a minimizer of $J_{\sigma,n}$.  
From Lemma \ref{l:pohozaev-Qp} it follows that 
the sharp constant in (\ref{c_gn-Qp}) 
is given by 
\begin{align}
 c_{GN}\inv=& J_{\sigma,n}(Q)\notag\\  
=& \frac{\lam\Vert Q_{\lam,1}\Vert_2^{2\sigma}}{ 2^{1-\sigma n/2} }
(2-\frac{\sigma n}{\sigma+1})  (\frac{\sigma n}{2\sigma+2-\sigma n})^{\sigma n/2}\,.\label{GN_half-bc} 
\end{align}
In particular,  if $p=1+4/n$, then 
we have
\begin{align}
&c_{GN}\inv= \frac{2\lam n}{n+2}\Vert Q_{\lam,1}\Vert_2^{4/n}
=\frac{2n}{n+2}\Vert Q_{1,1}\Vert_2^{4/n}\,, 
\label{e:c-GN-Qd}
\end{align}
 where 
note that  $u=Q_\lam=Q_{\lam,1}$ is the solution of (\ref{e:grstQ-masscri}).

Summarizing the above gives the following classical lemma. 
\begin{lemma}\label{l:sharp-GN-ineq-Q} Let $p\in (1,1+\frac{4}{n-2})$, $\sigma=(p-1)/2$ and   $\theta=
 {n}(\frac12-\frac{1}{p+1}) $.  Then
\begin{enumerate} 
\item[(a)]  
\begin{align}
&\norm{u}_{p+1}^{p+1}\le c_{GN} \norm{ u}_2^{(1-\theta)(p+1)}\norm{\nabla u}_2^{\theta(p+1)}  ,\label{GN:Up+1}
\end{align}
where 
   $c_{GN}$ is defined as in (\ref{GN_half-bc}):
\begin{align}
& c_{GN}\inv 
= \frac{\lam\Vert Q_{\lam,1}\Vert_2^{2\sigma}}{ 2^{1-\sigma n/2} }
(2-\frac{\sigma n}{\sigma+1})  (\frac{\sigma n}{2\sigma+2-\sigma n})^{\sigma n/2}\notag\\
=&\frac{\Vert Q_{1,1}\Vert_2^{2\sigma}}{ 2^{1-\sigma n/2} }
(2-\frac{\sigma n}{\sigma+1})  (\frac{\sigma n}{2\sigma+2-\sigma n})^{\sigma n/2}\,.\label{GN_half-1-1} 
\end{align} 
\item[(b)] In particular, if $p=1+4/n$, 
\begin{align}
&\norm{u}_{2+4/n}^{2+4/n}\le c_{GN}  \norm{ u}_2^{4/n} \norm{\nabla u}_2^2,\label{e:GN-deu-u2}
\end{align}
where $c_{GN}$ is given by  \eqref{e:c-GN-Qd}. 
\end{enumerate}
\end{lemma} 

Now  we are ready to prove Theorem \ref{t:gwp-blowup-Om}. 

 \begin{proof}[Proof of Theorem \ref{t:gwp-blowup-Om}] (1) Let $u\in C\left( [0, T_{max}), \sH^1\right) $ be the solution of equation \eqref{e:nls-H-rot}. 
Let $Q=Q_{\lam,1}$ be the g.s.s of (\ref{e:grstQ-masscri}).  From    (\ref{e:Eu-Vom}) and Lemma \ref{l:sharp-GN-ineq-Q} we have 
\begin{align*}
&\frac12\int |\nabla u|^2 +\frac{\ga^2}{2}\int |x|^2 |u|^2- E_{\Om,V} (u_0) +\ell_\Om(u_0)
\le\frac{1}{2}\norm{Q}_2^{-4/n} \norm{\nabla u}^2_2\norm{u}_2^{4/n}.
\end{align*}

Then with $\ga>0, \lam>0$  in \eqref{e:nls-H-rot}, we obtain 
\begin{align*}
&\frac12\left( 1-\big(\frac{\Vert  u_0\Vert_2}{\norm{ Q}_2}\big)^{4/n} \right)\int |\nabla u |^2+\frac{\ga^2}{2}\int |x|^2 |u |^2\\
\le& E_{\Om,V}(u_0)-\ell_\Om(u_0) .
\end{align*}
 Since $\Vert u_0\Vert_{2} <\Vert Q\Vert_{2} $, it follows that  
    $\norm{\nabla u(t)}_2 $ and $\norm{x u(t)}_2$  are uniformly bounded for all  $t \in [0, T_{max})$.  
Therefore $T_{max}=\iy$, in view of Proposition \ref{pr:gwp-lwpUav} (c), 
 which proves part (a). 


 (2) To prove the second part we consider the following initial data 
\begin{align}\label{Q:Calpha}
&u_0(x)= c\al^{n/2} Q(\al x)    
\end{align} with $ c\ge 1$, $\al>0$. 
 Then $\Vert u_0\Vert_{2}=  c\norm{Q}_2$. 
  It is easy to calculate  with $p=1+4/n$
  \begin{align*}
   &E_{\Om,V}(u_0)-\ell_\Om(u_0)
 \begin{cases}
=\frac{\ga^2}{2}J(0) & \text{if $c=1$},\\
<\frac{\ga^2}{2}J(0)&\text{if $c>1$}. 
\end{cases}  
 \end{align*}
  Indeed,  we have from \eqref{lowest-energy_Q}  
\begin{align*}
  &E_{\Om,V}(u_0)-\ell_\Om(u_0) =\frac12\int |\nabla u_0|^2+\frac{\ga^2}{2}\int |x|^2 |u_0|^2-\frac{2\lam}{p+1}\int |u_0|^{p+1}\\
 =& \frac{c^2\al^2}{2} \int |\nabla Q (x)|^2 + \frac{c^2\ga^2}{2\al^2}\int |x|^2 | Q(x)|^2
 -\frac{2\lam}{p+1} c^{p+1}\al^{2} \int |Q(x)|^{p+1}\\
 \le& c^2\al^2 E_{0,0}(Q) +\frac{c^2\ga^2}{2\al^2}\int |x|^2 | Q(x)|^2 =\frac{\ga^2}{2}J(0) ,\quad 
 \end{align*} 
where we note   $J(0)=\int |x|^2 |u_0|^2=\frac{c^2}{\al^2}\int |x|^2 |Q|^2$. 
 Therefore, according to Lemma \ref{l:J-J'-blup} (a),  
we conclude that  the solution $u(t)$ of (\ref{e:nls-H-rot}) corresponding to the given $u_0$
blows up on $[0,T_*)$ for some $T_*\in (0,\frac{3\pi}{4\ga})$. 
   \end{proof}

 The following proposition provides a blowup result 
as an application of  Lemma \ref{l:J-J'-blup} in the mass supercritical regime.
\begin{proposition}\label{p:blowup-masssup} 
Let $p\in (1+4/n, 1+4/(n-2))$ and $\lam>0$. Let $Q$ be the g.s.s in \eqref{e:grstQ-masscri}. 
 Then  there exists $u_0$ 
such that the corresponding solution 
$u$ of \eqref{e:nls-H-rot} blows up in finite time. 
\end{proposition}
\begin{proof} Take  
 \begin{align}\label{E:u0-enux2Q}
&u_0(x)=c e^{i\nu |x|^2}Q(x), \quad\text{$\nu\in\R$ and $ c>0$}.
 \end{align}
A straightforward calculation shows that  $J'(0)=
 4\nu c^2\int |x|^2|Q|^2$ and 
 \begin{align*} 
   E_{\Om,V}(u_0)-\ell_\Om(u_0)
 =& \frac{c^2}{2}(4\nu^2+\ga^2)\int |x|^2 |Q|^2+\frac{c^2}{2}\int |\nabla Q|^2-\frac{2\lam c^{p+1}}{p+1}\int |Q|^{p+1}\\
=:&h(c).
 \end{align*}
Let $c_0:=c_0(\nu,p)>0$ be the (unique) positive zero of $h(c)$. 
We have, if  $c>c_0$, then   
  \begin{align}\label{J0-J'(0)_supcri}
  E_{\Om,V} (u_0 )-\ell_\Om (u_0)<0\, ;
 \end{align}
and if $c=c_0$ and $\nu<0$, then $E_{\Om,V} (u_0 )-\ell_\Om (u_0)=0$ and $J'(0)<0$. 
We conclude that there exist finite time blowup solutions according to Lemma \ref{l:J-J'-blup} (c) and (d) respectively. 
\end{proof} 

\begin{remark}
 If $p=1+4/n$, $\norm{u_0}_2=\norm{Q}_2$, then 
 the so-called $\mathcal{R}$-transform, 
 yields all blowup solutions modular symmetries (scaling, translation and phase invariance)
 with initial data given by either $Q$ or for $\nu>0$ 
\begin{align}
& u_0(x)= \nu^{n/2}e^{i\nu} e^{-i \frac{\nu}{4} |x|^2}  Q(\nu x), 
\end{align} 
  see \cite[Subsection 4.1]{BHZ19a}. 
Note that $Q$  is the profile for blowup,   
while the standing waves  $u=e^{i t}Q_{\Om,\ga}$  apparently are global solutions, cf.  \eqref{EOmV:minimizer-c}.
\end{remark}
  

\begin{remark} 
In the case $\Om=0,\ga\ne 0$, 
  Carles \cite{Car02c} also tells that $Q$ is the critical minimal mass blowup.
However, the result is obtained by lens transform and Merle's characterization for the profile, 
while our proof relies on a virial identity for \eqref{e:nls-H-rot}, which seems more direct and simpler. 
In particular, the blowup condition in Lemma \ref{l:J-J'-blup}
is  sharper than those in \cite[Corollary 4.3 (ii)]{Car02c} for that case 
and 
 \cite{BHIM}  for 
 mNLS.  
\end{remark}

\setcounter{tocdepth}{2} %
\setcounter{subsection}{1}

\subsection{Minimal mass of blowup for \eqref{e:nls-H-rot} with $\lam>0$ is $\norm{Q}_2^2\,$ } Let $Q_{\Om,\ga,c}\in \Sigma$ be a ground state solution (g.s.s.) for the minimization problem: 
\begin{align}\label{E:minimizerQc}
E(Q_{\Om,\ga,c})=\inf \{ E_{\Om,V}(u): u\in \Sigma, \norm{u}_2^2=c^2>0 \},  
\end{align}
where $E_{\Om,V}$ is given as in (\ref{e:Eu-Vom}),
also see  \eqref{EOmV:minimizer-c} in Sections \ref{s:grst-exist-Energy} for the definition of g.s.s. in the setting of general nonlinearity.  
Let $p=1+4/n$ and $\lam>0$. 
Theorem \ref{t:gwp-blowup-Om} shows that $\norm{Q}_2^2$ is the threshold of minimal mass blowup if 
$Q$ is  the ground state of (\ref{e:grstQ-masscri}).
This is consistent with 
the profile description for blowup solutions 
\cite{LeNamRou18p}, where it is essentially shown that if $p=3$, $n=2$, then 
as $c^2\nearrow \norm{Q}_2^2$,  
the ground state $Q_{\Om,\ga,c}$ blowup  occurs at the lowest point of the trap $V$, 
that is, if $c=\norm{Q}_2$, there is no minimizer for (\ref{E:minimizerQc}). 
See also  the characterization of blowup profiles in \cite{BHZ19a,Car02c,GuoLY19,GuoSeir14}. 

\subsection{Profile for blowup solutions of focusing RNLS (\ref{e:nls-H-rot})} 
In view of \cite{BHZ19a,Car02c,Me93},   we are able  to apply similar argument and techniques to describe the behavior of minimal mass blowup solutions of \eqref{e:nls-H-rot} at the mass level $\norm{Q}_2$ in the mass-critical case 
by showing that all  blowup solutions  result from concentration of energy as $t\to T$. 
Let $u_0\in \sH^1$ and $\Vert u_0\Vert_2=\Vert Q\Vert_2$. If  $u\in C([0,T), \sH^1)$ is the blowup solution of \eqref{e:nls-H-rot} with $A=\Om\la -x_2,x_1,0\dots,0\ra$,  $V=\frac{\ga^2}{2}|x|^2$, then 
 there exist functions 
$x(t)\in \R^n $ and $\theta(t)\in\R$ so that the following holds in $\mathscr{H}^1$:
\begin{align*}
&u(t,x)= \frac{1}{\lam(t)^{n/2}}Q\left(\frac{x-x(t)}{\lam(t)}\right)e^{i\theta(t)}+o(1)\quad as\; t\to T.
\end{align*} 
Here $Q\in H^1(\R^n)$ is given by  \eqref{e:grstQ-masscri} and $\lam(t)=\Vert \nabla Q\Vert_2/\Vert \nabla u(t)\Vert_2 $.
However, such behavior of collapse is unstable at the ground state level, 
see \cite{BHZ19a,MR05}.  This is in consistence with the stability result at $p=1+4/n$ where $Q$ is the threshold,
as is shown in Theorem \ref{t:gwp-blowup-Om} and  Theorem \ref{t:stabi-masscr3}. 
 We shall prove Theorem \ref{t:stabi-masscr3} in Sections \ref{s:grst-exist-Energy} and \ref{s:orb-stab-u} 
  that when the initial data is below the ground state level, 
namely, $c<\norm{Q_{\lam,1}}_2$, 
then \eqref{e:nls-H-rot} is orbital stable. 
The analogous result is valid 
 for NLS (\ref{E:nls-Up}) with zero potential. 
Moreover, Theorem \ref{t:gwp-blowup-Om} shows 
 the solution  is strongly unstable when $M(u_0)=M(Q)$, where $\Vert \nabla u(t)\Vert_2\approx\frac{\Vert \nabla Q\Vert_2}{T-t}$ 
 as $t\to T$.  The  question concerning blowup rate  above the ground state level has been briefly discussed in (\ref{log-log_RNLS}), Section \ref{s:intro}.

\section{Virial identity for inhomogeneous RNLS}\label{s:Virial-inhom}
In this section, we  prove a virial identity for (\ref{e:psi-VomLz}), which 
will allow us to show a similar blowup result in the case of an inhomogeneous nonlinearity. 
In particular, we shall prove Theorem \ref{sharpQ:inhom}, where we determine
the sharp threshold of (\ref{e:psi-VomLz}) for global existence and blowup.  
Consider the inhomogeneous RNLS (\ref{e:psi-VomLz}), 
where $p\in [1,1+4/(n-2))$ and $\lam=\lam(x)$  satisfies the conditions in \mbox{Hypothesis \ref{mu-hypo}} throughout this section. 

\vspace{.20in}
\begin{hypothesis}\label{mu-hypo} Assume $\lam\in C^1(\R^n)$ satisfies the following:
\begin{enumerate}
 \item[(i)] $\lam$ is radially symmetric
\item[(ii)] $\lam_{min}\le \lam(x)\le \lam_{max}$ 
\item[(iii)] $x\cdot\nabla \lam(x)\le 0$, 
\end{enumerate} 
where we write $\lam_{min}=\inf_x \lam(x)\ge 0$ and $\lam_{max}=\sup_x\lam(x)>0$.
\end{hypothesis} 
Examples of $\lam's$ satisfying the hypothesis include, e.g. $\lam=\lam_0+\la x\ra^{-m}$, where $\lam_0>0$, $m\ge 0$.

 \subsection{Conservation laws  for equation (\ref{e:psi-VomLz})}
 Let  $p\in (1,1+4/(n-2) )$ and $\lam=\lam(x)$ as above. Let $V=\frac{\ga^2}{2}|x|^2$ and $ L_\Om= -\Om L_z$ be the same notation as for 
(\ref{e:nls-H-rot}). Suppose $u_0\in \Sigma$ and $u$ is the corresponding solution of \eqref{e:psi-VomLz}. 
Following the same line of proof as in the case $\lam$ being a constant, one can show 
the local wellposedness and blowup alternative theorem like Proposition \ref{pr:gwp-lwpUav}. 
Moreover, the mass, energy and angular momentum are conserved in the lifespan $I=(-T_{min},T_{max})$:
\begin{align}  
M(u)=&\int |u|^2=M(u_0)\notag\\
\td{E}_{\Om,V}(u)=&\frac12\int |\nabla u|^2+\int V | u |^2- \frac{2 }{p+1}\int \lambda(x) |u|^{p+1} +\ell_\Om( u)\label{E-lam(x)}\\
=&\td{E}_{\Om,V}(u_0) \notag\\
 \ell_\Om(u)=& -\Om\int \overline{u} {L_z} u= \ell_\Om(u_0).\notag
\end{align}  

\subsection{The $\td{J}$-functional} 
Define the $\td{J}$-functional  in $H^1\setminus \{0\}$ associated to $\lam$
 \begin{align}\label{J_sn:inhomo} 
\td{J}(u):=\td{J}_{\sigma,n,\lam}(u)=
\frac{ \Vert u\Vert_2^{2+2\sigma-n\sigma } \Vert \nabla u\Vert_2^{n\sigma}}{\Vert \td{\lam} u\Vert_{2\sigma+2}^{2\sigma+2}}\,,
\end{align}
where $\td{\lam}=\lam^{1/(p+1)}$,
cf. \eqref{J_sn[u]}.  Then it is easy to note that the minimum of $\td{J}$ exists, namely,
there exits  $\td{j}(\sigma,n,\lam)>0$ such that 
\begin{align}
& \td{j}(\sigma,n,\lam)=\min_{0\ne u\in H^1} \td{J}_{\sigma,n,\lam}(u).  
\end{align} 

\begin{lemma}\label{j_lam_max} Let  $p=1+2\sigma\in (1,1+\frac{4}{n-2})$. Let  $\lam=\lam(x)$ satisfy Hypothesis \ref{mu-hypo}. 
We have 
\begin{align}\label{sharpC:lam_max}
 \td{j}(\sigma, n,\lam) =\td{j}(\sigma, n,\lam_{max}) \,.
\end{align} 
\end{lemma}
\begin{proof} First, we have from the definition of $\td{J}$
\begin{align}\label{j:min-max}
&\td{j}(\sigma, n,\lam_{max})\le \td{j}(\sigma, n,\lam)\le \td{j}(\sigma, n,\lam_{min})\,.
\end{align}
Next, we show 
\begin{align}\label{j:lam_max}
&\td{j}(\sigma, n,\lam_{max})\ge \td{j}(\sigma, n,\lam)\,.
\end{align}
In fact, take a localizing sequence $\{u_\veps\}=\{ \frac{1}{\veps^{n/2}} u(\frac{x}{\veps})  \}$ around $x_0=0$,
where $\veps=\veps_k>0$ and 
 $u={Q}_{\lam_{max}}$ is the ground state solution of (\ref{e:grstQ-masscri}).   
Then, 
we observe that  
as $\veps\searrow 0$\,,
\begin{align*}
&\td{J}_{\sigma,n,\lam}(u_\veps) 
=\frac{ \Vert u\Vert_2^{2+2\sigma-n\sigma } \Vert \nabla u\Vert_2^{n\sigma}}{ \int \lam(\veps x) |u|^{2\sigma+2} }   \\
\to&\frac{ \Vert u\Vert_2^{2+2\sigma-n\sigma } \Vert \nabla u\Vert_2^{n\sigma}}{\lam_{max} \Vert u\Vert_{2\sigma+2}^{2\sigma+2}}\,
=\td{j}(\sigma, n,\lam_{max})\,,
\end{align*}
which proves (\ref{j:lam_max}). 
Combining (\ref{j:min-max}) and (\ref{j:lam_max})  yields (\ref{sharpC:lam_max}). 
\end{proof} 

A corollary of Lemma \ref{j_lam_max} follows. 
\begin{corollary}\label{c:GN:lam(x)} 
Let $p\in (1,1+\frac{4}{n-2})$. 
There is a  sharp constant $\td{c}_{GN}>0$ such that for all $u$ in $H^1$
\begin{align}\label{GN:lam(x)}
\int \lam(x)|u|^{p+1} \le \td{c}_{GN} \norm{ u}_2^{2-\frac{(p-1)(n-2)}{2}} \norm{\nabla u}_2^{n(p-1)/2} \,,
\end{align} 
where  the  constant $\td{c}_{GN}$ is given by 
 \begin{align}\label{c_gn-Qp}
& \td{c}\inv_{GN}=j(\sigma,n,\lam_{max}).
  \end{align} 
In particular, 
Note that  if $p=1+4/n$,  
\begin{align}\label{lam:c_GN}
&\td{c}_{GN}\inv=\lam_{max}\inv{c}_{GN} =\frac{2 n}{n+2}\Vert Q_{\lam_{max}}\Vert_2^{4/n}
\end{align}
by (\ref{GN:Up+1}) and \eqref{e:c-GN-Qd}.  
\end{corollary} 


\begin{remark} We have seen that the scaling invariance 
is lost for $\td{J}_{\sigma,n,\lam}$ when $\lam$ is non-constant. Such symmetry breaking has led to technical difficulties and non-existence of a minimizer for $\td{J}_{\sigma,n,\lam}$. 
At first one's attention is drawn on the plausible 
inequality $ \td{j}(\sigma, n,\lam) >\td{j}(\sigma, n,\lam_{max})$:  For $p=1+4/n$,  
\begin{align}
&  \frac{ \Vert \nabla Q_{\lam_{max}}\Vert_2^{4/n } \Vert Q_{\lam_{max}}\Vert_2^2}{\Vert\td{\lam}  Q_{\lam_{max}}\Vert_{2+4/n}^{2+4/n}}
> \frac{ \Vert \nabla Q_{\lam_{max}}\Vert_2^{4/n } \Vert Q_{\lam_{max}}\Vert_2^2}{\lam_{max} \Vert Q_{\lam_{max}}\Vert_{2+4/n}^{2+4/n}} \,.
\end{align}
But a limiting argument near the maximum point of $\lam$ 
shows that these two constants are identical, as is seen in the proof of Lemma \ref{j_lam_max}. 
{The moral tells that the inhomogeneous problem (\ref{J_sn:inhomo}) might be subtle and delicate}. 
 
 \end{remark}

\subsection{Localized virial identity for (\ref{e:psi-VomLz})}\label{ss:virial-mu(x)}
Let $\rho(x)=\rho(|x|) $ be radial. Define $J(t):=J_\rho(t)=\int \rho |u|^2$. 
By similar argument as in the proof of  Lemma \ref{l:virial-J(t)-Vom},
 we are able to show the localized virial identity for equation
(\ref{e:psi-VomLz}), whose detailed proof will be omitted.  
The magnetic versions were obtained in \cite{FV09,Gar12} and earlier  \cite{GRib91a} for homogeneous equations. 
\begin{lemma}\label{J:rho-u^pVOm} Let  $p\in [1,1+4/(n-2))$ and $\rho\in C_0^\iy(\R^n)$ be radially symmetric.
If $u$ is a  solution of \eqref{e:psi-VomLz} in  $C^1(I, \Sigma)$, then we have
	\begin{align*}
	&J_\rho'(t) =\Im \int \bar{u}\nabla\rho \cdot \nabla u 
	\end{align*}
and 	\begin{align}
	J_\rho''(t)=& -\frac14\int \Delta^2 \rho |u|^2
	-\frac{p-1}{p+1} \int \lam \Delta \rho |u|^{p+1}
	{+} \int\nabla \bar{u}\cdot \mathrm{Hessian}(\rho)\nabla u \notag\\  
		& -\gamma^2 \int x \cdot \nabla\rho |u|^2
	+ \frac{2}{p+1} \int  \nabla \lam\cdot\nabla\rho  |u|^{p+1}. 
	\end{align}
	\end{lemma}

Ansatz  $\rho=|x|^2$ with a limiting argument
gives us the virial identity in the following lemma. 
 \begin{lemma}\label{U:VOm-inhom} 
 Let $p\in [1,1+4/(n-2))$ and $u_0\in \Sigma$. Let  $u$ be the corresponding solution of \eqref{e:psi-VomLz} in $C^1(I,\Sigma)$. 
Then it holds 
\begin{align*}
J''(t)=& 2\int |\nabla u|^2 -2\gamma^2 \int |x|^2 |u|^2\notag\\  
&- 2n\frac{p-1}{p+1} \int \lam  |u|^{p+1}
	+ \frac{4}{p+1} \int x \cdot \nabla\lam |u|^{p+1}\\
=&4\td{E}_{\Om,V}(u_0)+4\Om\int \bar{u}_0(L_z u_0) -4\ga^2\int |x|^2|u|^2\\
&+\frac{2}{p+1}(4-n(p-1)) \int \lam(x)|u|^{p+1}+\frac{4}{p+1}\int x\cdot \nabla\lam |u|^{p+1} .
\end{align*}
\end{lemma}


An application of Lemma \ref{U:VOm-inhom} yields  the counterpart parts $(a)$-$(d)$ in Lemma \ref{l:J-J'-blup} for equation (\ref{e:psi-VomLz}). For our purpose   we state the blowup criterion in the mass-critical case only. 
\begin{lemma}\label{J-inhomRNLS} Let  $p=1+\frac{4}{n}$.
Suppose $0\ne u_0 \in \Sigma$ satisfies  $J'(0)\le 0$ and 
 $ \td{E}_{\Om,V} (u_0 )-\ell_\Om (u_0)\le \frac{\ga^2}{2}J(0)$, that is $\td{E}_{0,0}(u_0)\le 0$. 
Then, there exists $0<T_*<\iy$ such that 
the corresponding solution  $u$ of equation \eqref{e:psi-VomLz} satisfies  
\begin{align}\label{deU:RNLS}
&\Vert \nabla u(t)\Vert_2 \to\iy \quad as \; t\to T_*\,.
\end{align}
\end{lemma}
\begin{remark}
An example of $u_0$ satisfying the conditions  $J'(0)= 0$ and
 $\td{E}_{0,0}(u_0)\le 0$ is given by $u_0=Q_{\lam_{min}}$ when $\lam_{min}>0$. 
  Comparing with part (a) in Lemma \ref{l:J-J'-blup},
the extra condition $J'(0)\le 0$ is to ensure $t\in (0,\frac{\pi}{2\ga}]$, see Remark \ref{r:E00-T}. 
\end{remark}

The proof of  Lemma \ref{J-inhomRNLS} is similar to that for  Lemma \ref{l:J-J'-blup} (a). 
We only mention a technical point  concerning dealing with the  inhomogeneous term in the virial identity.  
Following the proof of  Lemma \ref{l:J-J'-blup}, 
instead of (\ref{e:J''+J-EL0p}) we arrive at 
\begin{align}\label{J''+J-inhom} 
 J''(t) +4\ga^2 J(t)=&4\td{E}_{\Om,V}(u_0)-4 \ell_\Om(u_0)+\frac{4}{p+1}\int x\cdot \nabla\lam |u|^{p+1} 
 \end{align}
by applying Lemma \ref{U:VOm-inhom}. 
An integral representation for $J(t)$ is given by
 \begin{align}
J(t) =& (J(0)-\frac{1}{\ga^2}(\td{E}_{\Om,V}(u_0)-\ell_\Om(u_0)) ) \cos 2\ga t + \frac{J'(0)}{2\ga} \sin 2\ga t \notag\\
&+\frac{1}{\ga^2} 
( \td{E}_{\Om,V}(u_0)-\ell_\Om(u_0)) + \int^t_0 \frac{\sin (2\ga(t - s))}{2\ga} f(s) ds\notag\\
=& C \sin( 2\ga t +\beta)+\frac{1}{\ga^2} (\td{E}_{\Om,V}(u_0)- \ell_\Om(u_0))+ \int^t_0 \frac{\sin (2\ga(t - s))}{2\ga} f(s),\label{Csin2ga-2ga}
\end{align} 
where the constant $C$ is given as in (\ref{J(t)-Csin}) and $f(t):=\frac{4}{p+1}\int x\cdot\nabla \lam |u(t,x)|^{p+1}$. 
Since $f(t)\le 0$ by condition (iii) in Hypothesis \ref{mu-hypo}, 
it follows that
 \begin{align}\label{J<Csin2gamma}
J(t)\le& C \sin( 2\ga t +\beta)+\frac{1}{\ga^2} (\td{E}_{\Om,V}(u_0)- \ell_\Om(u_0)).
\end{align} 
Then, this inequality implies  Lemma \ref{J-inhomRNLS}. 


\subsection{Proof of Theorem \ref{sharpQ:inhom}}  
Recall that $Q_\lam=Q_{\lam,1}\in \Sigma$ is the unique positive radial ground state solution of equation  \eqref{e:grstQ-masscri}. 
 Then by virtue of the scaling invariance (i)-(ii) in Lemma \ref{l:pohozaev-Qp}, 
 it is easy to observe that $Q_{\lam}(x) =\lam^{-1/(p-1)} Q_{1,1}( x)$ and so 
 $\norm{Q_\lam}_2^2=\lam^{-n/2} \norm{Q_{1,1}}_2^2$ if $p=1+4/n$. 
In this subsection,  we  prove  Theorem \ref{sharpQ:inhom} for RNLS \eqref{e:psi-VomLz} by 
applying (\ref{GN:lam(x)}) and Lemma \ref{J-inhomRNLS}. 
Let $p=1+4/n$. 
Theorem \ref{sharpQ:inhom} states  that the threshold for 
(\ref{e:psi-VomLz}) is given by $\norm{Q_{\lam_{max}}}$.   

\begin{proof}[Proof of (a) in Theorem \ref{sharpQ:inhom}] 
We will show that if 
$\norm{u_0}_2<\norm{Q_{\lam_{max}}}_2$\,,  then 
$\int |\nabla u|^2+  \int |x|^2 |u|^2\le C$ for all $t$. 
 From (\ref{E-lam(x)}) and (\ref{GN:lam(x)}) we have 
\begin{align}
& \td{E}_{\Om,V}(u_0)-\ell_{\Om}(u_0)= \int \left(\frac12 |\nabla u|^2+V|u|^2-\frac{2}{p+1}\lam(x)|u|^{2+\frac4n}\right) \label{eE:lam(x)}\\
\ge& \int \left(\frac12 |\nabla u|^2+\frac{\ga^2}{2} |x|^2 |u|^2\right)-\frac{2}{p+1}  \td{c}_{GN} \norm{u_0}_2^{4/n} \norm{\nabla u}_2^2 \notag\\
\ge&\ \frac12\left(1-\left(\frac{ \norm{u_0}_2}{\norm{Q_{\lam_{max}}}_2} \right)^{4/n}\right) \int|\nabla u|^2+\int\frac{\ga^2}{2} |x|^2 |u|^2\,. \notag
\end{align}
Thus it holds that for all $t$,
\begin{align*}
& \norm{u(t)}_\Sigma\le C:=C(\td{E}_{\Om,V}(u_0), \ell_\Om(u_0),n,\lam_{max},\ga) \,,
\end{align*} 
provided $\norm{u_0}_2<\norm{Q_{\lam_{max}}}_2 $\,. 



\bigskip
\noindent 
{\em Proof of (b) in Theorem  \ref{sharpQ:inhom}.} \ According to Lemma \ref{J-inhomRNLS}, we only need to check 
$J'(0)=0$ and $\td{E}_{0,0}(u_0)\le 0$.  
Let $u_{0,\al}(x)= C\al^{n/2} Q_{\lam_{max}}(\al x) $, 
where $\al>0$ and $C=(1+\veps_0)^{n/4}$ for arbitrary $\veps_0>0$.   
The first condition is readily verified since  $J'(0)=2\Im \int x\bar{u}_0  \cdot \nabla u_0=0 $ as soon as $u_0$ and $\nabla u_0$ are real-valued. 
Now we look at the second condition.  
Let $p=1+4/n$ and write by (\ref{E-lam(x)})
\begin{align*}
\td{E}_{0,0}(u_{0,\al})=& C^2\al^2 I_{\al}\,,
\end{align*}
where 
\begin{align*}
I_{\al}=&\left(\frac12 \int |\nabla {Q}_{\lam_{max}}|^2-\frac{2}{p+1}\int \lam(\frac{x}{\al}) |{Q}_{\lam_{max}}|^{p+1}\right)\\
-&\frac{2\veps_0}{p+1}\int \lam(\frac{x}{\al})|{Q}_{\lam_{max}}|^{p+1} .
\end{align*} 
Note that as $\al\to\iy$ we have by 
 (\ref{lowest-energy_Q}) 
\begin{align*}
I_\al\to &E_{0,0}(Q_\lam)-\frac{2\veps_0\lam_{max}}{p+1}\int  |{Q}_{\lam_{max}}|^{p+1}\\
=&-\frac{2\veps_0\lam_{max}}{p+1}\int  |{Q}_{\lam_{max}}|^{p+1} <0.  
\end{align*}
So, $I_\al<0$  as soon as $\al$ is sufficiently large. Thus we see that if taking $\al $ large, then 
$\td{E}_{00}(u_0)<0$.  This proves that for any $\veps>0$ there exists $u_0=u_{0,\al}$ for some $\al>0$ 
such  that  (\ref{e:psi-VomLz}) has finite time blowup solution 
with $\norm{u_0}_2=\norm{Q_{\lam_{max}}}_2+\veps$.
\end{proof}

\begin{remark}\label{re:Q} 
The threshold $c_0=\norm{Q_{\lam_{max}}}_2$ seems odd at first, since $\norm{\td{Q}}_2$ seems a more natural candidate 
due to the gap  $ \norm{Q_{\lam_{max}}}_2<\norm{Q_{\lam_{min}}}_2$,  
where $\wtd{Q}\in H^1$ obeys (\ref{tQ:inhomo}),
cf.  \cite{Gen12i,LiuZZhZ06,Me96} for related  results. 
However, the proof of Theorem \ref{sharpQ:inhom} along with the proof of Lemma \ref{j_lam_max} 
 shows that  this is precisely the case. 
 In the blowup setting, there occurs the mass concentration near $x=0$ where $\lam$ attains its maximum, which indicates 
 self-similarity for the blowup profile.
\end{remark} 

\begin{remark}\label{rk:gap-Qmax-min} 
 There remains an open question: Does there exist $u_0$ with $\norm{u_0}_2=\norm{Q_{\lam_{max}}}_2$ such that 
the  flow map $u_0\mapsto u$ blows up in finite time?  
This would require  further examination of   the  behavior of the solution  that is dependent on  the regularity of $\lam$ near a point where 
local maximum is assumed.   
 See \cite{BaCaDuy10,RaSz11jams} for the treatment on the minimal mass problem in the case of an  inhomogeneous NLS under a flatness condition. 
\end{remark}  

\begin{remark}\label{r:geometry}   
The result in Theorem \ref{sharpQ:inhom} might have independent interest. 
It has connections with focusing NLS on a manifold $\mathcal{M}$, 
\begin{align}\label{psi:M} 
i \psi_t =-\frac12\De_{\mathcal{M}} \psi- |\psi|^{\frac4{n}} \psi \,  
\end{align} 
where $\De_{\mathcal{M}}$ is the Laplace-Beltrami operator on $\mathcal{M}$. 
For instance, if 
$\mathcal{M} =\mathbb{H}^n$, $n=2$ denotes the hyperbolic space, then 
\begin{align*} 
&  \De_{\mathcal{M}}=  \pa_r^2+  \frac{\phi'(r)}{\phi(r)} \pa_r+ \frac{1}{\phi(r)^2}\pa^2_\theta \,,
\end{align*} 
 where  
the Riemannian metric is given by $ds^2= dr^2 +\phi(r)^2 d\theta^2$
 with $\phi(r)=
 \sinh r$.  
In the radial case, the transform $\psi(t,r)= u(t,r)  (r/\phi(r))^{1/2}$ converts \eqref{psi:M} to 
\begin{align}
i u_t =-\frac12\De u +V u - \lam(r)|u|^2 u 
\end{align} 
where   $V$ is a quadratic function of $r$  
and $\lam(r)=r/\phi(r)$ is derived  from the weight or group-theoretical translation-invariant 
measure  on $\mathcal{M}$ for the group $G=SL(2,\R)$.
One remarkable application is that  
 our proof of Theorem 1.2 can be applied  to obtain a sharp threshold for the 
case  
$\lam(x)=\frac{r}{\sinh r}$, where $\lam $ satisfies Hypothesis \ref{mu-hypo}.
Note that such    
$\lam$ fails to satisfy the {\em flatness condition} $\lam''(0)=0$ at the critical point $x_0=0$
as is required in \cite[Theorem 1.1]{BaCaDuy10},  
 thus the construction of a minimal mass blowup solution does not work through a pseudo-conformal transform method, cf. \cite[Section 1.3]{BaCaDuy10}. 
However, the problem of the existence of minimal mass blowup solution at exact\textcolor{red}{ly} the ground state level 
seems still open, as is mentioned in Remark \ref{rk:gap-Qmax-min}.
\end{remark}



\section{Ground state solutions}\label{s:grst-exist-Energy} 
In this section, we will prove the existence for the set of ground states for (\ref{e:nls-H-rot}) in the case  $|\Om|<\ga$.
In fact,  we will elaborate the result for more general nonlinearity as given in the following RNLS on $\R^{1+n}$
 \begin{align}\label{e:nls-G-Omga}
	i u_t = -\frac12\De u +\frac{\ga^2}{2}|x|^2 u -\kappa N(x,|u|)u
	- \Om L_z u,  \quad   u(0)= u_0 \in \Sigma\,.
	\end{align}
 The associated energy functional $E_{\Om,\ga}$ is given by 
	\begin{equation}
	E_{\Omega, \gamma}(u)
	=\int \left(\frac12 |\nabla u|^2 +\frac{\ga^2}{2} |x|^2 |u|^2 -\kappa G(|u|^2)- \Om \bar{u} L_z u \right) dx,\label{Eomga:u-G}
	\end{equation}
where $\kappa=\pm 1$ and $G(x,s^2)=2\int_0^{s} N(x,\tau)d\tau$ for $s\ge 0$.
 Throughout this section and next, we assume conditions $(\Om,\ga)_0$ and $(G)_0$ given in the following hypothesis. 
\begin{hypothesis}\label{hy:G0-Om-ga} Let $\Om$, $\ga$ and $G$ satisfy the following conditions. 
	\begin{enumerate}
	\item[($\Omega, \gamma)_0$:]\label{it:om-ga} $|\Omega| < \gamma$.
	\item[($G_0$):]\label{it:G0} Let $G: \R^n\times \mathbb{R}_+ \rightarrow \mathbb{R}_+$ be continuous and differentiable 
	such that
	   for some constant $C>0$ and  $ p>1$,
\begin{align}\label{G(xv)^p}
0\le G(x,v) \leq C ( v+v^{\frac{p+1}{2}}). 
\end{align}
	\end{enumerate}
\end{hypothesis}
Note that  the nonlinearity $N(x,|u|)u=\pa_sG(x,|u|^2)u$  includes $\mu(x) |u|^{p-1}u$, 
and the combined inhomogeneous terms 
$\mu_1(x)|u|^{p_1-1}u+\mu_2(x) |u|^{p_2-1}u$, 
$\mu, \mu_1, \mu_2$ being positive and bounded functions.
In what follows we will use the abbreviation $G(s)=G(x,s)$ 
unless otherwise specified. 

Recall   $\Sigma= \{ u \in H^1: \int |x|^2 |u|^2 < \infty \}$ is a Hilbert space endowed with the inner product 
\begin{align*}
& \la u, v\ra_{\Sigma}=\la H_{\Om,\ga}u,v\ra_{L^2},  
\end{align*}
where the norm is given by  $\| u \|_\Sigma^2 = \la H_{\Om,\ga}u, u\ra
\approx \| \nabla u \|_2^2 + \| x u \|_2^2+\|u \|_2^2$	 and $H_{\Om,\ga}$ is defined as in (\ref{HomV}). 
Then we have $\Sigma $ is compactly embedded in $L^r$ for all $r\in [2,\iy)$ 
in view of \cite[Lemma 5.1]{Zh00}.
\begin{definition}\label{def:Ic-Q_OmV} Let 
	\begin{equation}
	I_c = \inf \{ E_{\Omega, \gamma}(u): u \in S_c \},
	\quad
	S_c = \{ u \in \Sigma: \int |u|^2 = c^2 \}.
	\label{(3.4)}
	\end{equation}
Then $u\in S_c$ is called a ground state solution (g.s.s.) provided $u$ is a solution to the following minimization problem 
\begin{align}\label{EOmV:minimizer-c}
E_{\Om,\ga}(u)=I_c\,. 
\end{align} 
\end{definition} 

Note that any solution of (\ref{EOmV:minimizer-c}) solves the Euler-Lagrange equation 
	\begin{equation}
	-\om u = -\frac12\De u + \frac{\ga^2}{2} |x|^2 u -\kappa N(x,|u|)u 
	- \Om L_z u\,, 
	\label{e:3.5}
	\end{equation}
where $-\om\in\R$ is a  Lagrange multiplier. 
Once we construct a g.s.s. $u=Q_{\Om,\ga}:=Q_{\Om,\ga,c}\,$, then $e^{i\om t}Q_{\Om,\ga}$ 
is a solitary wave solution to \eqref{e:nls-G-Omga}.
Write
\begin{align}\label{E1-G}
&E_{\Om, \ga}(u)=E^1_{\Om, \ga}(u)-\kappa\int G(|u|^2) .
\end{align}
Then  $E^1_{\Om, \ga}$ denotes the linear energy 
	\begin{equation*}
	\begin{aligned}
	E^1_{\Om, \ga}(u)
	& =\frac12\| \nabla_A u \|_2^2 + \int V_e |u|^2\,,  
	\end{aligned}
	\end{equation*}
where  $A$ is given in (\ref{A:Om-x}) and $V_e=\frac12(\ga^2-\Om^2)(x_1^2+x_2^2)+\frac{\ga^2}{2}x_3^2+\cdots+\frac{\ga^2}{2}x_n^2$ as in Section \ref{s:prelimin}.  
It follows by the assumption $(\Om, \ga)_0$ and \cite[Proposition 4.4]{Z12a} that for all $u\in\Sigma$
	\begin{align}
	&E_{\Omega, \gamma}^1 (u)=\la H_{\Om,\ga}u, u\ra\approx \norm{u}^2_{\Sigma}.  \label{(3.7)} \end{align}
Now we  state the first result in this section.
\begin{theorem}\label{t:mini-E-Ic}  Assume Hypothesis \ref{hy:G0-Om-ga}. 
  Let $c>0$. 
\begin{enumerate}
\item[(a)] 
If  $1\le p<1+\frac{4}{n}$ and $\kappa=1$, then the  
problem (\ref{EOmV:minimizer-c}) has a minimizer, 
i.e.,  there exists $u:=Q_{\Om,\ga} \in S_c$ such that $E_{\Om, \ga}(u) = I_c$.
\item[(b)] If $1\le p<1+\frac{4}{n-2}$ and $\kappa=-1$, 
 then (\ref{EOmV:minimizer-c}) always has a minimizer solution. 
  \end{enumerate}
\end{theorem}

\begin{proof}  (a) Let   $p\in (1,1+\frac{4}{n})$ and $G$ satisfy $(G_0)$.  Then $0<\sigma n<2$,
where $\sigma=(p-1)/2$.  
  First we show that for all $u$ in $S_c$, 
 $E_{\Om,\ga}(u)\approx \norm{u}_{\Sigma}^2-\int G(|u|^2) dx$ is bounded from below, which implies  
\begin{align}
I_c\ne -\iy. \label{Ic-below}
\end{align}
Indeed,  we obtain from (\ref{G(xv)^p}) and \eqref{GN:Up+1} and  Young's inequality
\begin{align}
&\int G(|u|^2) \le C\int |u|^2+C\int |u|^{p+1}\notag\\
\le& C\norm{u}_2^2+C \frac{  \lVert u\rVert_2^{(2\sigma+2-\sigma n)q  }}{ \veps^{q}}
+C\veps^{q'}  \lVert \nabla u\rVert_2^{2}\,,\quad \forall \veps>0, \label{Cq-deU}
\end{align}
where we set $a= \veps\inv \lVert u\rVert_2^{2\sigma+2-\sigma n}$, $b=\veps\norm{\nabla u}_2^{\sigma n}$,
$q'=\frac{2}{\sigma n}$, $q=\frac{2}{2-\sigma n}$ in the inequality
$ab\le \frac{a^{q}}{q}+\frac{b^{q'}}{q'}$. 
Then, by taking $\veps$ sufficiently small,
it is easy to see from (\ref{E1-G}) and (\ref{Cq-deU}) that there exists some $C_0>0$ such that for all $u$ in $S_c$
\begin{align}
&E_{\Om,\ga}(u)\ge -C_0 .
\end{align}

Secondly, we show that  any given minimizing sequence $\{u_n\}\subset S_c$ of (\ref{EOmV:minimizer-c}) 
 is relatively compact in $\Sigma$. 
To do this we prove the following properties:
\begin{enumerate}
\item[(i)]  $\lVert  u_n\rVert_\Sigma$ is bounded.
\item[(ii)] There exists a subsequence $\{u_{n_j}\}$ weakly converging to some $w\in \Sigma$ such that 
$\lim_{j \to \iy} \int G(|u_{n_j}|^2)= \int G(|w|^2) $. 
\item[(iii)] The limit $w$ is a minimizer solution to (\ref{EOmV:minimizer-c}).
\item[(iv)] The sequence $\{u_{n_j}\}$ converges to $w$  in the $\Sigma$-norm.
\end{enumerate}

Since  $\{u_n\}$ is a minimizing sequence in $S_c$\,, $E_{\Om,\ga}(u_n)$ are bounded. 
It \mbox{follows} from (\ref{E1-G}) and (\ref{Cq-deU}) that
\begin{align*}
&E^1_{\Om,\ga}(u_n)=E_{\Om,\ga}(u_n) +\int G(|u_n|^2)\\
\le& C+\frac{  C}{ \veps^{q}}+C\veps^{q'} 
 \lVert \nabla u_n\rVert_2^{2}.
\end{align*}
By taking $\veps$ sufficiently small, we obtain, in view of \eqref{(3.7)},
\begin{align*}
& \lVert  u_n\rVert_\Sigma^{2}\approx  E^1_{\Om,\ga} (u_n)\le C.
\end{align*}
This proves (i).

The property  (i) suggests that $\{u_n\}$ is weakly precompact in  $\Sigma$, 
namely,  there exists $w \in \Sigma$ such that, up to a subsequence,  
$\{u_n\} \rightharpoonup w$ weakly in $\Sigma$. 
Since $\Sigma$ is compactly embedded in $L^r$ for any $r \in [2, \frac{2n}{n-2})$, 
we see that there exists  
a subsequence,  which is still denoted by $\{u_n\}$, such that 
	\begin{equation}
	u_n \rightarrow w \; \text{in} \; L^r,
	\quad \forall \ 2\le r <  \frac{2n}{n-2}\,. \label{(3.8)}
	\end{equation} 
 

 Claim. \begin{equation}\label{int-Gun-w}
\int G(|u_n|^2) dx \rightarrow \int G(|w|^2) . \end{equation} 
Indeed, since $u_n \rightarrow w$ in $L^2$ and $L^{p+1}$,
by the continuity of $G$ we have, up to a subsequence, for almost all $x$ 
\begin{align} 
&\lim_n G(|u_n(x)|^2)=G(|w(x)|^2). \label{e:Gun-v(x)} 
\end{align}
Moreover, by the inverse 
dominated convergence theorem, 
we can find a subsequence $\{u_{n_j}\}$ of $\{u_n\}$ 
and two functions $h_1$ and $h_2$ such that: 
\begin{enumerate}
\item $h_1 \in L^2$ and $h_2 \in L^{p+1}$
\item $|u_{n_j}(x)| \leq h_1(x)$ and $|u_{n_j}(x)| \leq h_2(x)$ a.e. 
\end{enumerate}
Thus $G(|u_{n_j}(x)|^2) \leq C(h_1(x)^2 + h_2(x)^{p+1})$.  
Since $h_1^2 + h_2^{p+1} \in L^1$, 
applying the dominated convergence theorem, we obtain 
	\begin{equation}
	\lim_{j \to \iy} \int G(|u_{n_j}(x)|^2) 
	= \int G(|w(x)|^2) .
	\label{(3.9)}
	\end{equation}
This actually implies that the whole sequence in the Claim verifies (\ref{int-Gun-w}), 
because otherwise one can easily finds a contradiction by a subsequence argument. 
Thus we have proven property (ii).

To show (iii), we note that,  by the lower semi-continuity of $\| \cdot \|_{\Sigma}$ for  weakly convergence sequence, 
	\begin{equation}
	\| w \|_{\Sigma}^2 \leq \liminf \| u_n \|_{ \Sigma}^2\,.
	\label{(3.10)}
	\end{equation}
Now combining \eqref{(3.8)},   \eqref{int-Gun-w} and \eqref{(3.10)} we obtain
	\begin{equation}
	\begin{cases}
	\| w \|_2^2 = c^2, \\
	E_{\Omega, \gamma} (w) \leq \liminf E_{\Omega, \gamma} (u_n) = I_c\,.
	\label{e:(3.11)}
	\end{cases}
	\end{equation}
Consequently $w$ is a minimizer of \eqref{(3.4)}. 
In turn, since $w$ is a weak limit of ${u_n}$ in $\Sigma$, plus $\lim_n \norm{u_n}_\Sigma=\norm{w}_\Sigma$, 
we conclude that $\{u_n\}$ converges to $w$ in $\Sigma$.
This proves (iv) and the existence of a ground state in the focusing case. 


(b) 
 Let $p\in(1, 1+\frac{4}{n-2})$ and $\kappa=-1$ (defocusing).  
We have 
\begin{align}\label{EOmV+G}
&E_{\Om,\ga}(u)\approx\norm{u}^2_{\Sigma}+\int G(|u|^2).
\end{align}
Then the problem \eqref{(3.4)} always has $0< I_c<\iy$.  
Let $\{u_n\} \subset S_c$ be a minimizing sequence of \eqref{EOmV:minimizer-c}: 
$\norm{u_n}_2^2 = c^2$ and $\displaystyle \lim_{n} E_{\Omega, \gamma}(u_n) = I_c$.

It follows from  (\ref{EOmV+G}) that $\norm{u_n}^2_{\Sigma}\lesssim E_{\Om,\ga}(u_n)$. 
Therefore $\{u_n\}$ is bounded in $\Sigma$. The remaining of the proof proceeds the same as in the focusing case.\end{proof} 

From the proof of Theorem \ref{t:mini-E-Ic}, we have the following corollary. 
\begin{corollary}\label{c:mini-relcomp-defoc} 
Under the assumptions  in either (a) or (b) of Theorem \ref{t:mini-E-Ic},
any minimizing sequence of (\ref{EOmV:minimizer-c}) is relatively compact in $\Sigma$. 
\end{corollary}

\begin{remark} If $N(x,|u|)=\mu(x)|u|^{p-1}$ and $u=Q_{\Om,\ga}$ is a g.s.s. of (\ref{EOmV:minimizer-c}), then 
 \begin{align*}
&-\om c^2= I_c-\frac{p-1}{p+1}\int \mu(x)|Q_{\Om,\ga}|^{p+1} .
\end{align*}
This identity can provide information on the signs of the $\om$ and $I_c$.
\end{remark}

\begin{remark} The convergence for the integrals in \eqref{int-Gun-w} essentially relies on 
the so called inverse dominated convergence theorem.  
 We can also prove \eqref{int-Gun-w} by the fact that 
any given $f_n\to f$ a.e. satisfying $\{f_n\}$ being uniformly absolute continuous in the sense of integration, then 
 $\lim_n \int f_n \to \int f$ 
 holds (without recourse to the Lebesgue dominated convergence theorem).
\end{remark}

\subsection{Mass-critical case} 
When  $p= 1+4/n$, $\kappa=1$, the minimization problem (\ref{EOmV:minimizer-c})  
 still makes sense for  small values of $c>0$. 
We show in Theorem \ref{th:grstQ-masscri} that there exists  a threshold mass level concerning the existence of g.s.s.
A priori, we need the following lemma by 
  the Gagliardo-Nirenberg inequality \eqref{e:GN-deu-u2} and the diamagnetic inequality 
\begin{align}
&\Vert \nabla |u| \Vert_{2} \le \Vert (\nabla -iA)u\Vert_2\,, \label{deU<grad-iAu}
\end{align}
which can be found in e.g. \cite{EL89} or \cite{BHIM}.

\begin{lemma}\label{l:diamagnGN-ineq-Qa} 
 There is a sharp constant $c_{GN}$ such that for all $u\in \Sigma$ 
\begin{align}&\norm{u}_{2+4/n}^{2+4/n}\le c_{GN} \norm{\nabla_A u}_2^2 \norm{ u}_2^{4/n} ,\label{e:GN-UA-diamag}
\end{align}
\end{lemma}
where $c_{GN}=\frac{n+2}{2n \norm{Q_{1,1}}_2^{4/n}}=\frac{n+2}{2\lam n \norm{Q_{\lam,1}}_2^{4/n}}$ for any $\lam>0$. 

\begin{theorem}\label{th:grstQ-masscri} Let $p=1+\frac4n$ and $\kappa=1$. 
Let $\Om$, $\ga$ and $G$ verify Hypothesis \ref{hy:G0-Om-ga}. 
Let $Q_{\lam,1}$ be the positive radial  solution of \eqref{e:grstQ-masscri},
where  $\lam=\frac{n+2}{n}C$ 
and $C$ is the constant in \eqref{G(xv)^p}. 
We have: 
\begin{enumerate}
\item[(a)] If $c<\norm{Q_{\lam,1}}_2$, then there exists a ground state solution of Problem \eqref{EOmV:minimizer-c}. 
Moreover, any minimizing sequence for \eqref{EOmV:minimizer-c} is relatively compact in $\Sigma$. 
\item[(b)] In particular,  let $G(v)=\frac{\lam_0 n}{n+2}v^{1+2/n}$ for any given $\lam_0>0$,  then Problem \eqref{EOmV:minimizer-c}
corresponds to the problem of finding the g.s.s for \eqref{U:OmV}. 
It holds that if $c<\norm{Q_{\lam_0,1}}_2$,    
then there exists  a g.s.s. $u=Q_{\Om,\ga}$ of  \eqref{EOmV:minimizer-c} satisfying for some $\om=\om_c\in\R$
\begin{align}\label{U:OmV}
&-\frac12\De u+\frac{\ga^2}{2} |x|^2u-\lam_0 |u|^{\frac{4}{n}}u-\Om L_zu=-\om u,\qquad \norm{u}_2=c\,.\end{align}
\end{enumerate}
\end{theorem}

\begin{proof}
From (\ref{G(xv)^p}) and Lemma \ref{l:diamagnGN-ineq-Qa}  we have for all $u$ in $S_c$
	\begin{align}
	E_{\Omega, \gamma}(u)=&\frac12\int |\nabla_A u|^2+ \int V_e |u|^2-\int G(|u|^2)\notag\\
	 \geq& \frac12\int |\nabla_A u|^2+ \int V_e |u|^2-\frac{\lam n}{n+2}\int |u|^2
	 -\frac{\lam n}{n+2} \, c_{GN}\norm{u}_2^{\frac{4}{n}}\norm{\nabla_A u}_2^2\notag\\
	=& \frac12\left(1-\frac{\norm{u}_2^{4/n}}{\norm{Q_{\lam,1}}_2^{4/n}}\right) \int |\nabla_A u|^2 
	+ \int V_e |u|^2 -C_0.\label{EOmV>K}
	\end{align}
Like  in the proof of Theorem \ref{t:mini-E-Ic}, we have, given $ c<\norm{Q_{\lam,1}}_2$, then the above estimate 
shows:
\begin{enumerate}\label{en:Ec-lowbd-masscri}
\item[(i)] For all $u\in S_c$, $E^1_{\Om, \ga}(u)$ is bounded from below by a constant. Hence $I_c$ exists. 
\item[(ii)] 
Let $\{u_n\}\subset S_c$ be a minimizing sequence for  $(\ref{EOmV:minimizer-c})$.
Then it is bounded in $\Sigma$. 
\item[(iii)] Since $\Sigma$ compactly embedded in $L^r$, $r\in [2,\frac{2n}{n-2})$, there exists $w\in\Sigma$
and a subsequence, still denoted by $\{u_n\}$, such that
$\int G(|u_n|^2)\to \int G(|w|^2)$.
\item[(iv)] Then, (iii) implies that up to a subsequence, $E_{\Om,\ga}(w)=\lim_n E_{\Om,\ga}(u_n)$
and $u_n\to w$ in $\Sigma$. In particular, $w$ is a minimizer. 
\end{enumerate}   
Thus the existence result (iv) proves part (a) of the theorem. Part (b) is a special case of part (a).  
\end{proof}

\begin{remark} The theorem suggests that in the focusing, mass-critical case, there exists ground state solutions on the 
mass level set   $S_c$ that are below the free ground state $Q_{\lam,1}$. 
However, such solutions are   not unique, consult e.g. \cite{EL89}.
If $c>\norm{Q}_2$, then  there may not exist a solution for (\ref{EOmV:minimizer-c}), see \cite{GuoSeir14} 
 for the case where $V$ is a quadratic potential and  $\Om=0$.
Compare also \cite{ANenS18} where the case of anisotropic $V$ is considered for $p<1+4/n$, $n=2,3$.
\end{remark}
 
\section{Orbital stability of standing waves}\label{s:orb-stab-u}
In this section, we  prove the orbital stability of the standing waves of \eqref{e:nls-G-Omga}. 
Throughout this section we will assume $G$ satisfies:
 \begin{enumerate}
	\item[($\Omega, \gamma)_0$:]\label{Om-ga} $|\Omega| < \gamma$.
	\item[($G_0^\prime$):]\label{dG0} Let $G: \R^n\times \mathbb{R}_+ \rightarrow \mathbb{R}_+$ be continuous and differentiable such that 
	   for some constant $C>0$ and 	$1\le p< 1+\frac{4}{n-2}$, 
\begin{align}
|\pa_vG(x,v)| \leq C( 1+v^{\frac{p-1}{2}}). 
\end{align}
\end{enumerate} 
Note that ($G_0^\prime$) is slightly stronger than $(G_0)$. 
It is easy to see from the condition $(G^\prime_0)$ and the embedding $H^1\hr L^{r}$, $r\in [2,\frac{2n}{n-2})$ 
\begin{equation}
	E_{\Om, \ga}: \Sigma \rightarrow \mathbb{R} \quad
	\text{is continuous} . \label{(3.3)}
	\end{equation}

 \begin{proposition}\label{pr:dG-EU} Let $G$ satisfy the conditions $(\Om,\ga)_0$ and ($G'_0$) for $p>1$ and 
 $\kappa=\pm 1$. 
 Let $r=p+1$ and  $q=\frac{2(p+1)}{(p-1)}$. 
\begin{enumerate}
\item[(a)] Given any $u_0$ in $\Sigma$, 
there exists a unique maximal solution $u$ of (\ref{e:nls-G-Omga}) in $C(I, \Sigma)\cap L^q(I, H^{1,r})\cap C^1(I, \Sigma\inv)$ on $ I=(-T_{min},T_{max})$ such that 
if $T_{max}$ (or $T_{min}$) is finite, then $\norm{u(t)}_\Sigma\to \iy$ as $t\to T_{max}$, (respectively $t\to -T_{min}$).
Moreover, the following conservation laws hold on $I$: 
\begin{align*} 
&\norm{u}_2=\norm{u_0}_2\\
&E_{\Om,\ga}(u)=E_{\Om,\ga}(u_0). 
\end{align*} 
\item[(b)] The solution $u$ of (\ref{e:nls-G-Omga}) is global in time if one of the following holds:
\begin{enumerate}
\item[(i)]  $\kappa=1$, $p=1+\frac{4}{n}$ and $\norm{u_0}_2<\norm{Q_{\lam,1}}_2$, 
where $\lam$ is given as in Theorem \ref{th:grstQ-masscri}.
\item[(ii)] $\kappa=-1$ and $1<p<1+\frac{4}{n-2}$.
\end{enumerate}
\end{enumerate}
\end{proposition}
The proof of the local theory above is similar to that of Proposition \ref{pr:gwp-lwpUav} and Theorem \ref{t:gwp-blowup-Om} (a), hence omitted. 
 Also, see  \cite[Theorem 1]{CazE88} for the case $|\Om|=\ga$. 
The global existence follows from the estimate 
$\norm{u(t,\cdot)}_\Sigma\le C\norm{u_0}_\Sigma $ in view of \eqref{EOmV>K} in the case $\kappa=1$ and 
$\norm{u_0}_2<\norm{Q_{\lam,1}}_2$. 
 As we will see, the stability  requires the local existence and uniqueness for $u_0\in\sH^1$ along with
  the mass and energy conservation laws in the proposition above.  

 For $c>0$ let 
\begin{equation}\label{e:Zc-minEomV}
 Z_c = \{u \in \Sigma: E_{\Om, \ga}(u) = I_c, \norm{u}_2=c \}, 
\end{equation}
where $I_c$ is the minimum given by \eqref{(3.4)}. Then $Z_c$ is the set of all ground state solutions
for (\ref{EOmV:minimizer-c}). 
\begin{definition}\label{def1:Zc-orbstabi} 
  The set $Z_c$ is called  orbitally stable if given any $\vphi \in Z_c$, the following property holds:
   $\forall \varepsilon > 0$,
$\exists \delta > 0$ such that 
	\begin{equation}
\| u_0 -\vphi \|_{\Sigma} < \delta\To\sup_t\inf_{\phi \in Z_c} \| u(t, \cdot) - \phi \|_{\Sigma} < \varepsilon,\label{e:orbit-stabi-zc}
	\end{equation}  
where $u_0\mapsto u(t, \cdot)$ is the unique solution flow map  of (\ref{e:nls-G-Omga}). 
\end{definition}




\begin{theorem}\label{t:orb-stab-Zc} 
Assume the condition $(\Omega, \gamma)_0$ 
and ($G^\prime_0$). Let $c>0$. Suppose one of the following is satisfied: 
\begin{enumerate}
\item[(a)] $\kappa=1$, $1<p<1+\frac{4}{n}$;
\item[(b)] $\kappa=1$, $p=1+\frac{4}{n}$ and $c<\norm{Q_{\lam,1}}_2$;
\item[(c)] $\kappa=-1$, $1<p<1+\frac{4}{n-2}$\,.
\end{enumerate} 
Then $Z_c$ is orbitally stable for \eqref{e:nls-G-Omga}. 
In particular,  if $G(v)=\frac{2\lam}{p+1}v^{(p+1)/2}$ for $v\ge 0$ with $\kappa=1$, we obtain 
Theorem \ref{t:stabi-masscr3}.
\end{theorem}

\begin{proof} From  Theorem \ref{t:mini-E-Ic} 
and Theorem \ref{th:grstQ-masscri} we know  $Z_c \ne\emptyset$. 
 According to Proposition \ref{pr:dG-EU}, we have 
\begin{enumerate}
\item[(i)] RNLS \eqref{e:nls-G-Omga} has global in time solution, 
that is, the lifespan is 
 $(-\iy,\iy)$; 
\item[(ii)] RNLS \eqref{e:nls-G-Omga} enjoys the conservation laws for mass and energy. 
\end{enumerate}
In order to prove the orbital stability, we argue by contradiction. 
Suppose $Z_c$ is not stable, then  there exists $w_0 \in Z_c$,
$\varepsilon_0 > 0$ and a sequence $\{ \psi_{0,n} \} \subset \Sigma$ such that
	\begin{equation}
	\| \psi_{0,n} - w_0 \|_\Sigma \rightarrow 0 \ \text{as} \ n \rightarrow \infty \
	\text{but} \ \inf_{\phi \in Z_c} \| \psi_n(t_n, \cdot) - \phi \|_\Sigma \geq \varepsilon_0
	\label{e:(3.13)}
	\end{equation}
for some sequence $\{ t_n \} \subset \mathbb{R}$,
where $\psi_n$ is the solution of \eqref{e:nls-G-Omga} corresponding to the initial data $\psi_{0,n}$.

Now let $z_n = \psi_n(t_n, \cdot)$.  
Since $w_0\in S_c$ and $E_{\Omega, \gamma}(w_0) = I_c$,
it follows from the continuity of $\| \cdot \|_2$ and the continuity of $E_{\Omega, \gamma}$ on $\Sigma$,
that $\| \psi_{0,n} \|_2^2 \rightarrow c^2$ and $E_{\Omega, \gamma}(\psi_{0,n}) \rightarrow I_c$.
By the conservation of the mass $\| \cdot \|_2$ and the energy $E_{\Omega, \gamma}$,
we certainly have  $\| z_n \|_2^2 \rightarrow c^2$ and $E_{\Omega, \gamma} (z_n) \rightarrow I_c$.
Therefore,  $\{ z_n \}$ is a minimizing sequence. According to Corollary \ref{c:mini-relcomp-defoc} and 
Theorem \ref{th:grstQ-masscri}, under any one of the hypotheses in  (a), (b) or (c), 
$\{z_n\}$
contains a subsequence $\{z_{n_j}\} $ converging  to an element $w$ in $\Sigma$ as $j\to\iy$, such that 
$\| w \|_2^2 = c^2$ and $E_{\Omega, \gamma} (w) = I_c$. 
Thus  $w \in Z_c$, 
and as a consequence 
\[
\inf_{\phi\in Z_c} \| \phi_{n_j}(t_{n_j}, \cdot) - \phi \|_\Sigma \leq \| z_{n_j} - w \|_\Sigma\to 0,\] 
as $j\to \iy$. This contradicts \eqref{e:(3.13)}. 
Therefore, we have established the orbital stability. \end{proof} 

\begin{remark} Concerning the usual nonlinearity $N(x,|u|)u=|u|^{p-1}u$ with $\kappa=1$,
if $p=1+\frac{4}{n}$, $c\ge\norm{Q}_2$, or $p>1+\frac{4}{n}$,  then strong instability, namely, blowup can occur for (\ref{e:nls-H-rot}), 
see  Theorem \ref{t:gwp-blowup-Om}, Lemma \ref{l:J-J'-blup} and \cite{AMS12,BHIM}. 
\end{remark} 

\begin{remark} The stability of ground state solution for mNLS in $\R^3$ was initially considered in \cite{CazE88} 
and certain instability  was studied in \cite{GRib91a} and \cite{FuOh03i} for the mass-supercritical case. 
Our  treatment is motivated by  \cite{CazLion82} 
and \cite{HiStu04} on the concentration compactness method and the extension to NLS with a general nonlinearity given by (\ref{G(xv)^p}) 
by showing any minimizing sequence is relatively compact in  $\Sigma$.  
Comparing with \cite{CazE88} the  major improvement of
our results are the following:
\begin{enumerate}
\item[(i)]  We treat the case  $V_e=V-\frac{|A|^2}{2}\ne 0$ in all dimensions, which extends  the results in \cite{CazE88,EL89} where $V_e=0$. Note that this extension requires a norm characterization $\norm{u}_\Sigma\approx \norm{\nabla_A u}_2+\norm{x u}_2$, 
which was obtained in \cite{Z12a} and entails a non-trivial proof using harmonic analysis. 

\item[(ii)] On a remarkable level, we prove a sharp threshold result on the stability in the  mass-critical case for \eqref{e:nls-G-Omga},  
which can be viewed as a complement on the mass-subcritical result in \cite{CazE88} or \cite{ANenS18} if $n=2, 3$.  
 To prove this we have used a sharp diamagnetic Gagliardo-Nirenberg type inequality in \mbox{Lemma \ref{l:diamagnGN-ineq-Qa}.}  
\end{enumerate}\end{remark}

\subsection{Instability when $|\Om|>\ga$: A counterexample} 
In this subsection,  we provide an example in $\R^2$ and $\R^3 $ to show 
 non-existence of minimizers  for \eqref{e:Zc-minEomV}  if $|\Omega| > \gamma$  (fast rotating). 
For $\R^2$ our consideration is 
 motivated by 
 \cite{BuRo99}, \cite{Sei02},  
 and \cite{BaoWaMar05}. 
 This counterexample shows there does not exist a ground state solution to \eqref{EOmV:minimizer-c} 
  in either focusing or  defocusing case. 
 Note that in physics the state of the system is stable only when it attains its lowest energy.
So physically, there exhibits instability in the fast-rotating scenario, in consistence with  the numerical simulations. 

Let $\Om>\ga>0 $ and let $\kappa=1$, $G(v)= K(v+v^{a/2}) $, $a\in (2,\iy)$.  
Define a sequence of vortex (test) functions  $\{\psi_m\}\subset\Sigma$ for $m\in\Z$ 
using $x=(r,\theta)$, the polar coordinate, 
\begin{align}\label{psi:2d}
&\psi_m(x)=\frac{\ga^{(\frac{|m|+1}{2})}}{\sqrt{\pi (|m|)!}} |x|^{|m|} e^{-\ga |x|^2/2} e^{im\theta}.
\end{align}
Then $\norm{\psi_m}_2=1$ for all $m$. We have  $E_{\Om,\ga}(\psi_m)\to -\iy$ as $m\to +\iy$. 

Indeed, we compute by (\ref{Eomga:u-G}) 
 \begin{align*}
E_{\Om,\ga}(\psi_m)=& \frac12\int |\nabla\psi_m|^2+\int V|\psi_m|^2-\Om\int \overline{\psi_m}L_z(\psi_m) -\int G(|\psi_m|^2)\\
=& \frac{|m|+1}{2}\ga+\frac{|m|+1}{2}\ga-\Om m -K -2\pi K\int_0^\iy |\psi_m|^a rdr \\
=&(|m|+1)\ga-\Om m -K+o(1), \quad
\end{align*}
where it is easy to calculate for $ L_z= -i\pa_\theta$, 
 $V(x)=\frac{\ga^2}{2}|x|^2$ and $a>2$
\begin{align*}
&\int  |\nabla \psi_m|^2 =(|m|+1)\ga\notag\\
&\int V |\psi_m|^2  
 =\frac{|m|+1}{2}\ga \notag\\
&\int \overline{\psi_m} L_z (\psi_m)  
= m\\  
&\int |\psi_m|^a = \ga^{\frac{a}{2}-1} \pi^{\frac{3}{2}(1-\frac{a}{2})}  \frac{1}{2^{(\frac{a}{4}-1)}\sqrt{a} (\sqrt{|m|})^{\frac{a}{2}-1}} +o(1)\; \; \;\text{as $|m|\to \iy$.} 
\end{align*} 
\begin{remark} 
From the calculations we see that  $V$ may augment the kinetic energy, 
while  $L_\Om=-\Om L_z$ may add to
\begin{align*}
\begin{cases} 
\text{negative angular momentum energy} &\text{as}\; m\to +\iy\\
\text{positive angular momentum energy}  &\text{as}\; m\to -\iy
\end{cases}  
\end{align*}
and  have larger effect over the nonlinearity in the fast rotating regime.   
\end{remark}

\begin{remark} Let $|\Om|>\ga$.
We can also construct a counterexample in $\R^3$ as follows. Define using (\ref{psi:2d})  
 \begin{align*}
&\phi_m(x)= \al\psi_m(\al {x_1},\al {x_2}) h_0(x_3)\,, 
\end{align*} 
where $\al= (\frac{|m|}{\ga+\veps} )^{1/3}$, $0<\veps<|\Om|-\ga$ 
and $h_0(x_3)=\pi^{-1/4} e^{-x_3^2/2} $. 
Then, similar calculations lead to $E_{\Om,\ga}(\phi_m)\to -\iy$ as $m\,\sign(\Om)\to +\iy$. 
The advantage of such construction is that an easy modification 
allows one to extend it to the case of an anisotropic quadratic potential in $\R^n$.  
\end{remark}

\section{Concluding remarks} 
We conclude our paper with some remarks regarding the applicability and extension of our study to other related problems. 
This paper mainly addresses the threshold for the blowup 
 of  mass critical focusing RNLS (\ref{e:nls-H-rot}).  
 We are able to solve a difficult problem using refined method and techniques that might have some implications 
on the study of the threshold dynamics theory that generally includes the blowup rate, blowup profile and  location of the blowup points of the solutions 
as well as the stability {\em near the threshold}. 
These would provide precise descriptions of the blowup solutions near certain ground state level, 
 in consistence with the experimental evidence in physics lab owing to the principle that the state of the particle  maintains stable when it attains the lowest energy.  


1.  Equations (\ref{e:nls-H-rot}) and (\ref{e:psi-VomLz})  have extensions to general rotational NLS by reformulating   the term $L_\Om=iA\cdot \nabla$, where $A=M x$ with $x\in \R^n$ and $M$ being any skew-symmetric matrix \cite{BaoCai13,BHZ19a,Gar12}. The results in this paper, e.g., \mbox{Theorem \ref{t:gwp-blowup-Om}} and Theorem 
\ref{sharpQ:inhom},  
extend to such RNLS following the same proofs. 
We wish to emphasize that although some (non-bordering) results can be obtained from those of (\ref{E:nls-Up})
by means of the $\mathcal{R}$-transform, a pseudo-conformal or lens type  transform introduced in \cite{BHZ19a,Car11time}
certain crucial tools and properties, e.g., Lemma \ref{l:J-J'-blup} and the blowup of the endpoint $Q$-profile, 
cannot be obtained this way.  
Indeed, 
the solution of the variance $J(t)$ in (\ref{Csin2ga-2ga}) 
reveals that the standard NLS 
is the limiting case of the RNLS as $\ga\to 0$, but that limiting case does not directly provide an ultimate threshold information for the blowup profile. 
In addition, the $\mathcal{R}$-transform may apply to the case $p=1+4/n$ only, since when $p>1+4/n$,  $\mathcal{R}$ does not preserve the energy. 

2. In the mass supercritical case $p\in (1+\frac{4}{n}, 1+\frac{4}{n-2})$, the blowup dynamics appears  subtle 
 for $\norm{u_0}_2=\norm{Q}_2$, where $Q$ is the unique ground state of (\ref{e:grstQ-masscri}).  
Within an arbitrarily small neighborhood of $Q$, there always exist $\vphi_0$ and $\psi_0$ such that the flow 
$\vphi_0\mapsto \vphi$ blowups in finite time, and, $\psi_0\mapsto \psi$ exists globally in time in $\sH^1$.
We have seen that the criteria in Lemma \ref{l:J-J'-blup} 
give sufficient blowup conditions  if $p>1+4/n$. 
It would be of interest to further investigate the {\em sharp} threshold problem in the mass-supercritical regime. 
Moreover,  one should be able to adapt the method in this paper to prove 
 the analogue for some nonlocal problems such as  RNLS with a Hartree nonlinearity.  
 In  the energy-critical regime, one can consider (\ref{e:nls-H-rot}), e.g., in the case 
 where the power nonlinearity becomes exponential type  in two dimensions \cite{CIMM09}.


3. 
Technically, in oder to give a deeper sharp description  of the singularity formation and the local
asymptotic stability of the ``self-similar'' blow-up profile  for $p\ge 1+4/n$, 
it would require  
  computing the trajectory of the solution on the soliton manifold \cite{
  MarMeRa14,Wein85}. 
In our setting, it is natural and easier to consider the blowup and stability on a submanifold $\mathcal{O}$ of $\sH^1$.
In $\R^3$, 
 $\mathcal{O}$ can be taken as  the set of all the $m$-vortex solitons of the form  
$Q^{(m)}(r,\theta,z)=e^{im \theta} \phi(r,z) $ 
using cylindrical coordinates \cite{CazE88,EL89,SimZw11}.  
We would need to analyze the spectral properties of the associated  linearization operators  $L_+^{(m)}=H_{\Om,V}+ 1- p |Q^{(m)}|^{p-1} $ and $L_-^{(m)}=H_{\Om,V}+ 1-  |Q^{(m)}|^{p-1} $, 
which may lead to solving the profile equation of (\ref{e:nls-H-rot}). 

4. The method and results on the orbital stability and the construction of the ground state solutions of (\ref{e:nls-H-rot}) and (\ref{e:nls-G-Omga}) apply to focusing and defocusing cases. Such treatment  might enable one to  further study  the vortex soliton stability for the Pauli-Schr\"odinger system as well as mNLS with more insight.  

5. Finally, we raise the question on the relativistic spin-$\frac12$ model NLD \cite{BouC19,PeShi14}:
 Determine the threshold for the blowup and  
stability for the nonlinear Dirac equation in higher dimensions? 

\vs{.28in}
\nd {\bf Acknowledgments} This work was initiated when S.Z. visited the NYU institute in Shanghai during summer 2015.  Further, S. Z. would like to thank \mbox{W. A. Strauss}  
and J. Holmer for helpful  comments during  his visit at Brown University fall 2015. He also thanks the hospitality and support of the \mbox{Department} of Mathematics at Brown for an enjoyable environment.  The authors gratefully thank the anonymous referees for constructive and valuable comments that have helped improving the technical presentation of the manuscript in the current version. 


\end{document}